\documentclass{aic}

\usepackage{amsthm}
\usepackage{cleveref}
\usepackage{caption}
\usepackage{subcaption}
\usepackage{ccicons}
\usepackage{floatrow}
\usepackage{enumitem}
\usepackage[square,sort,comma,numbers]{natbib}

\newtheorem*{thm*}{Theorem}
\newtheorem{thm}{Theorem}
\Crefname{thm}{Theorem}{Theorems}

\newtheorem*{lem*}{Lemma}
\newtheorem{lem}[thm]{Lemma}
\Crefname{lem}{Lemma}{Lemmas}

\newtheorem*{claim*}{Claim}
\newtheorem{claim}{Claim}
\crefname{claim}{Claim}{Claims}
\Crefname{claim}{Claim}{Claims}

\newtheorem{prop}[thm]{Proposition}
\Crefname{prop}{Proposition}{Propositions}

\crefname{cor}{Corollary}{Corollaries}

\crefname{conj}{Conjecture}{Conjectures}

\Crefname{qn}{Question}{Questions}

\Crefname{obs}{Observation}{Observations}

\Crefname{ex}{Example}{Examples}

\theoremstyle{definition}

\Crefname{prob}{Problem}{Problems}

\Crefname{defn}{Definition}{Definitions}

\theoremstyle{remark}

\renewenvironment{proof}[1][]{\begin{trivlist}
\item[\hspace{\labelsep}{\bf\noindent Proof#1.\/}] }{\qed\end{trivlist}}

\newenvironment{proof_claim}[1][]{\begin{trivlist}
\item[\hspace{\labelsep}{\emph{\noindent Proof#1.}\/}] }{\qed\end{trivlist}}

\newcommand{\ceil}[1]{
    \left\lceil #1 \right\rceil
}
\newcommand{\floor}[1]{
    \left\lfloor #1 \right\rfloor
}

\newcommand{\eps}{\varepsilon}
\newcommand{\E}{\mathbb{E}}

\usepackage{tikz}
\usepackage{mathdots}
\usepackage{xcolor}
\usetikzlibrary{calc}
\usetikzlibrary{decorations.pathreplacing}
\usetikzlibrary{positioning,patterns}
\usetikzlibrary{arrows,shapes,positioning}
\usetikzlibrary{decorations.markings}

\def \smvx {circle[radius = .07][fill = black]}

\tikzstyle{edge}=[very thick]
\definecolor{bostonuniversityred}{rgb}{0.8, 0.0, 0.0}
\definecolor{arsenic}{rgb}{0.23, 0.27, 0.29}
\tikzstyle{diredge}=[postaction={decorate,decoration={markings,
		mark=at position .65 with {\arrow[scale = 1]{stealth};}}}]

\tikzstyle{diredge2}=[postaction={decorate,decoration={markings,
		mark=at position .55 with {\arrow[scale = 1]{stealth};}}}]

\newcommand{\defPt}[3]{
	\def \pt {(#1, #2)}
	\coordinate [at = \pt, name = #3];
}

\tikzset{
   K5/.pic={
     \foreach \x in {1,...,5}{%
    \pgfmathparse{(\x-1)*360/5}
    \node[circle,fill=black, inner sep=1 pt] (N-\x) at (\pgfmathresult:1 cm) [thick] {};
  }
  
  \foreach \x in {1,...,4}{%
    \foreach \y in {\x,...,5}{%
        \path (N-\x) edge[line width=1 pt] (N-\y);
  }
  }
  }
}

\tikzset{
   K55/.pic={
     \defPt{0}{-0.9}{e0}
     \defPt{0}{-0.3}{e1}
     \defPt{0}{0.3}{e2}
     \defPt{0}{0.9}{e3}

    \defPt{1.8}{-1.05}{f0}
     \defPt{1.8}{-0.35}{f1}
     \defPt{1.8}{0.35}{f2}
     \defPt{1.8}{1.05}{f3}

\foreach \i in {0,...,3}
    \foreach \j in {0,...,3}
        \draw[red, line width=1pt ] (e\i) -- (f\j);

\fitellipsis{$(e0)-(0,0.1)$}{$(e3)+(0,0.1)$}{0.6};
\pic[rotate=-90, xscale=0.225,yscale=0.45, red] at (0,0) {K4s};

\fitellipsis{$(f0)-(0,0.1)$}{$(f3)+(0,0.1)$}{0.6};
\pic[rotate=-90, xscale=0.225,yscale=0.5, blue] at (1.8,0) {K5};

  }
}

\tikzset{
   K4s/.pic={
     \foreach \x in {1,...,4}{%
    \pgfmathparse{(\x-1)*360/4}
    \node[draw,circle,fill=black, black, inner sep=0.7 pt] (N-\x) at (\pgfmathresult:1 cm) [thick] {};
  }
  
  \foreach \x in {1,...,3}{%
    \foreach \y in {\x,...,4}{%
        \draw [thick] (N-\x) -- (N-\y);
  }
  }
  }
}

\tikzset{
   K4/.pic={
     \foreach \x in {1,...,4}{%
    \pgfmathparse{(\x-1)*360/4}
    \node[draw,circle,fill=black, inner sep=1 pt] (N-\x) at (\pgfmathresult:1 cm) [thick] {};
  }
  
  \foreach \x in {1,...,3}{%
    \foreach \y in {\x,...,4}{%
        \draw (N-\x) -- (N-\y);
  }
  }
  }
}

\tikzset{
   K4t/.pic={
     \foreach \x in {1,...,4}{%
    \pgfmathparse{(\x-1)*360/4}
    \node[draw,circle,fill=black, inner sep=1 pt] (N-\x) at (\pgfmathresult:1 cm) [thick] {};
  }
  
  \foreach \x in {1,...,3}{%
    \foreach \y in {\x,...,4}{%
        \path (N-\x) edge[color=red, line width=1.25 pt] (N-\y);
  }
  }
  }
}

\newcommand{\fitellipsis}[3] 
{\draw [fill=white] let \p1=(#1), \p2=(#2), \n1={atan2(\y2-\y1,\x2-\x1)}, \n2={veclen(\y2-\y1,\x2-\x1)}
    in ($ (\p1)!0.5!(\p2) $) ellipse [ x radius=\n2/2+0.1cm, y radius=#3cm, rotate=\n1];
}

\newcommand{\fitellipsisnfill}[3] 
{\draw [] let \p1=(#1), \p2=(#2), \n1={atan2(\y2-\y1,\x2-\x1)}, \n2={veclen(\y2-\y1,\x2-\x1)}
    in ($ (\p1)!0.5!(\p2) $) ellipse [ x radius=\n2/2+0.1cm, y radius=#3cm, rotate=\n1];
}

\newcommand{\fitellipsisrfill}[3] 
{\draw [fill=red,fill opacity=0.1] let \p1=(#1), \p2=(#2), \n1={atan2(\y2-\y1,\x2-\x1)}, \n2={veclen(\y2-\y1,\x2-\x1)}
    in ($ (\p1)!0.5!(\p2) $) ellipse [ x radius=\n2/2+0.1cm, y radius=#3cm, rotate=\n1];
}

\newcommand{\fitellipsisbfill}[3] 
{\draw [fill=blue,fill opacity=0.1] let \p1=(#1), \p2=(#2), \n1={atan2(\y2-\y1,\x2-\x1)}, \n2={veclen(\y2-\y1,\x2-\x1)}
    in ($ (\p1)!0.5!(\p2) $) ellipse [ x radius=\n2/2+0.1cm, y radius=#3cm, rotate=\n1];
}

\aicAUTHORdetails{%
  title = {Tight Ramsey Bounds for Multiple Copies of a Graph}, 
  author = {Matija Buci\'c and Benny Sudakov},
  plaintextauthor = {Matija Bucic and Benny Sudakov},
  plaintexttitle = {Tight Ramsey Bounds for Multiple Copies of a Graph}, 
}   

\aicEDITORdetails{%
   year={2023},
   number={1},
   received={22 October 2021},   
   published={3 February 2023},  
   doi={10.19086/aic.2023.1},      
}   

\begin{document}

\begin{frontmatter}[classification=text]

\title{Tight Ramsey Bounds for \\ Multiple Copies of a Graph} 

\author[mbuc]{Matija Buci\'c\thanks{Research supported in part by NSF grant CCF-1900460.}}
\author[bsud]{Benny Sudakov\thanks{Research supported in part by SNSF grant 200021\_196965.}}

\begin{abstract}
The Ramsey number $r(G)$ of a graph $G$ is the smallest integer $n$ such that any $2$ colouring of the edges of a clique on $n$ vertices contains a monochromatic copy of $G$. Determining the Ramsey number of $G$ is a central problem of Ramsey theory with long and illustrious history. Despite this there are precious few classes of graphs $G$ for which the value of $r(G)$ is known exactly. One such family consists of large vertex disjoint unions of a fixed graph $H$, we denote such a graph, consisting of $n$ copies of $H$ by $nH$. This classical result was proved by Burr, Erd\H{o}s and Spencer in 1975, who showed $r(nH)=(2|H|-\alpha(H))n+c$, for some $c=c(H)$, provided $n$ is large enough. Since it did not follow from their arguments, Burr, Erd\H{o}s and Spencer further asked to determine the  number of copies we need to take in order to see this long term behaviour and the value of $c$. More than $30$ years ago Burr gave a way of determining $c(H)$, which only applies when the number of copies $n$ is triple exponential in $|H|$. In this paper we give an essentially tight answer to this very old problem of Burr, Erd\H{o}s and Spencer by showing that the long term behaviour occurs already when the number of copies is single exponential.
\end{abstract}
\end{frontmatter}

\section{Introduction}
   
     Ramsey theory refers to a large body of mathematical results, which all roughly speaking say that any sufficiently large structure is guaranteed to have a large well--organised substructure. For example, the celebrated theorem of Ramsey \cite{ramsey1929problem} from 1929 says that for any fixed graph $G$, every $2$-edge-colouring of a sufficiently large complete graph contains a monochromatic copy of $H$. The \emph{Ramsey number of $H$}, denoted $r(H)$, is defined to be the smallest order of a complete graph satisfying this property. 
    
    Determining $r(H)$ is one of the central problems of Ramsey theory with long and illustrious history. Perhaps the most famous instance of the problem is when $H$ is a complete graph, where despite considerable attention over the years \cite{ramsey1,ramsey2,ramsey3,ramsey4,ramsey-ub} the best known bounds, up to improvements of the lower order terms in the exponents, remain $2^{n/2}\le r(K_n) \le 2^{2n}$ first obtained by Erd\H{o}s in 1947 \cite{ramsey-lb} and Erd\H{o}s and Szekeres in 1935 \cite{ramsey-ub} respectively. This situation is indicative of the situation in general, there are precious few families of graphs for which the Ramsey numbers are known precisely, see \cite{path-ramsey,cycle-ramsey,cycle-ramsey2} for some notable exceptions and \cite{dynamic-ramsey} for a dynamic survey on what is known about Ramsey numbers for various classical families of graphs.
    
    A particular family where $r(G)$ is, at least in some sense known, is when $G$ consists of a vertex disjoint union of many copies of some small graph $H$, we will denote such a graph consisting of $n$ vertex disjoint copies of $H$ by $nH$. Ramsey numbers of $nH$ were first considered by Burr, Erd\H{o}s and Spencer in 1975 \cite{bes} when they showed the following result. Here $\alpha(H)$ stands for the largest size of an independent set in $H$. 
    
    \begin{thm}[\cite{bes}]\label{thm:bes} For any graph $H$ without isolated vertices there is a constant $c=c(H)$ and $n_0$ such that if $n \ge n_0$
    $$r(nH)= (2|H|-\alpha(H))n+c.$$
    \end{thm}
    
    While this result on one hand offers quite a precise understanding on how the answer behaves, on the other hand already Burr, Erd\H{o}s and Spencer pointed out two major deficiencies of their result. The first one is that their argument provides no way of actually computing the constant $c(H)$ and only provides them with a double exponential upper bound in $k=|H|$ on $c(H).$ The second one is that apart from knowing it is finite they obtain no bound on $n_0$, whatsoever. They naturally asked for the situation here to be remedied.
    
    This led to a slew of results which determined or estimated both $c$ and $n_0$ for a number of natural but specific families of graphs, see Section 5.13 of \cite{dynamic-ramsey} for a survey of what is known. Perhaps the most well-known of these special cases is the problem of determining $r(nK_3)$, in which case the answer is $5n$ provided $n \ge 3$ so $c(K_3)=0$ and one may already take $n_0=3$.
    
    We will be interested in the original, general problem of Burr, Erd\H{o}s and Spencer. Here, the state of the art was due to Burr \cite{burr1} who, more than 30 years ago, gave an indirect, but finite, way of computing $c(H)$ settling in a sense the first of the above mentioned deficiencies of \Cref{thm:bes}. He also obtained a bound on when the transition to the long term behaviour occurs, namely he showed one may take $n_0$ to be only triple exponential in $k=|H|$. Taking into account the fact that in most known special cases this transition occurs very quickly he mentioned it would be very desirable to show something of the sort happens in general as well. Our main result establishes precisely such a result, giving an essentially tight answer to the second problem of Burr, Erd\H{o}s and Spencer from their 1975 paper.
    
    \begin{thm}\label{thm:main} There exists $C>0$ such that for any $k$-vertex graph $H$ without isolated vertices there is a constant $c=c(H)$ such that if $n \ge 2^{Ck}$ then

    $$r(nH)= (2|H|-\alpha(H))n+c.$$
    \end{thm}
    
    Our approach allows for a similar, slightly cumbersome, way of determining $c(H)$ as that of Burr. Due to this we postpone further discussion of this to \Cref{sec:sym} and will only illustrate here with a particularly nice example.
    
    \begin{thm}\label{thm:clique}
    There exists $C>0$ such that for any $n \ge 2^{Ck}$ we have 
    $$r(nK_k)=(2k-1)n+r(K_{k-1})-2.$$
    \end{thm}
    
    This example also serves to illustrate that our \Cref{thm:main} is essentially tight, namely if $n$ is subexponential the standard bounds, mentioned above, on Ramsey numbers $r(K_k)$ tells us that there are infinitely many values of $k$ for which the RHS of the above expression is smaller than $r(K_k)$ so we might not be able to find even a singe copy of $K_k$ let alone $n$ of them. It is quite likely 
    that this should hold for all $k$ but estimating the difference between $r(K_{k-1})$ and $r(K_k)$ and showing one is at least a constant factor larger than the other for all $k$ is an old open problem in Ramsey theory \cite{difference-ramsey}.
    
    Our second result concerns an asymmetric Ramsey problem of an arbitrary graph $G$ against $nH$. In more generality $r(G,H)$ is defined to be the smallest order of a complete graph with the property that in any $2$-colouring of its edges we can find a copy of $G$ in the first colour or a copy of $H$ in the second colour. The fairly natural question of determining $r(G,nH)$ was raised by Burr who observed that it played a key role in the proof of \Cref{thm:bes} due to Burr, Erd\H{o}s and Spencer. It also plays a key intermediate role in his further work on determining $r(nH)$. Burr managed to determine $r(G,nH)$ (as a function of $G,H$) if $n$ is at least double exponential in $k=\max(|G|,|H|)$ and naturally asked if this can be improved. We once again obtain an essentially tight result answering this question. Here there is a nice explicit formula which requires a further extension of the definition of Ramsey numbers. Given two families of graphs $\mathcal{G},\mathcal{H}$ we define $r(\mathcal{G},\mathcal{H})$ to be the smallest order of the complete graph such that in any $2$-colouring of its edges we can find a copy of a graph from $\mathcal{G}$ in the first colour or a copy of a graph from $\mathcal{H}$ in the other colour.

    \begin{thm}\label{thm:asym}
		There exists $C>0$ such that for any non-empty graph $G$ and a connected graph $H$ with $k=\max(|G|,|H|)$ we have for any $n \ge 2^{Ck}$
		$$r(G,nH)= n|H|+r(\mathcal{D}(G),H)-1,$$
		where, $\mathcal{D}(G)$ stands for the family of graphs obtained by removing a maximal independent set from $G$.
    \end{thm}
    
    The requirement for $H$ to be connected is easily seen to be necessary for this tight bound to hold and our methods do allow for a similar result even if $H$ is disconnected. We discuss this point further in \Cref{sec:asym}.

\section{Preliminaries.}    
In this section we give a few easy lemmas which will come in handy. The first one is an adaptation of one of the standard proofs of Ramsey's theorem and allows us to say that graphs with small independent set have still fairly large subgraphs with extremely high minimum degree.

\begin{lem}\label{lem:ramsey}
Given an $n$-vertex graph $G$ with $\alpha(G) < k$ and $d \ge 2$, provided $n \ge 3 \cdot d^{k-1}$ we can find a subgraph of order $m \ge n/d^{k-1}$ with degree at least $(1-1/d)m$.
\end{lem}
\begin{proof}
If we can find a vertex in $G$ which has non-degree at least $n/d-1$ we select this vertex, delete all its neighbours and repeat within the remainder, where in each iteration $n$ is replaced with the number of vertices of the current graph. This process needs to stop after less than $k-1$ iterations as otherwise the $k-1$ vertices we selected together with any of the remaining $n/d^{k-1}-\frac{1}{d^{k-1}}-\ldots-1 \ge n/d^{k-1}-2\ge 1$ vertices span an independent set of size $k$. When the process stops the remaining graph has $m \ge n/d^{k-2}-\frac{1}{d^{k-2}}-\ldots-1\ge n/d^{k-2}-2\ge n/d^{k-1}$ vertices and every vertex has more than $(1-1/d)m$ neighbours, as desired.
\end{proof}

We will in most cases use the above lemma on $2$-coloured graphs in which we know that there is no monochromatic copy of some graph on $k$ vertices in some colour and then apply the lemma to the subgraph consisting of the edges of the other colour.

The following lemma will be used several times in our proofs. Before stating it let us define a \emph{$(k,\ell)$-join} to be a $2$ coloured graph whose vertices are split into sets $R$ and $B$ of sizes $k$ and $\ell$ resp. such that all edges inside $R$ are red, all edges inside $B$ are blue and edges in between $R$ and $B$ are of the same colour. We refer to $R$ as the red part and $B$ as the blue part of the join. We say the join is red if the edges between $R$ and $B$ are red and we say it is blue otherwise.

\begin{lem}\label{lem:join}
Given a $2$-coloured complete graph and a partition of its vertices into sets $R$ and $B$ both of size at least $2^{5k}$ such that $R$ has no blue $K_k$ and $B$ has no red $K_k$ we can find a $(k,k)$-join with red part in $R$ and blue part in $B$. 
\end{lem}
\begin{proof}
Since there is no red $K_k$ inside $B$, by the standard Ramsey estimate we can find a blue $K_{4k}$ inside $B$, which we call $B'$. Now consider the majority colour between $R$ and $B'$, w.l.o.g. say red. Let $d(v)$ denote the red degree for $v\in R$ towards $B'$. Then $\sum_{v \in R} d(v) \ge |R||B'|/2=2|R|k$. On the other hand the number of red stars with $k$ leaves in $B'$ and centre in $R$ is 
$$\sum_{v \in R} \binom{d(v)}{k} \ge |R|\binom{2k}{k},$$
using Jensen's inequality. In particular, since there are $\binom{4k}{k}$ choices for the leaf set this means there is a set $R'$ of at least $|R| \binom{2k}{k}/\binom{4k}{k} \ge |R|/3^k$ vertices in $R$ joined in red to some fixed set of $k$ vertices inside our $B'$ (which spans a blue clique). Since $|R'| \ge |R|/3^k \ge 2^{2k}$ and there is no blue copy of $K_k$ in $R$, so also in $R'$, the usual Ramsey estimate tells us there is a red copy of $K_k$ in $R'$. This $K_k$ together with the blue $K_k$ in their common red neighbourhood make our desired join.
\end{proof}

The final preliminary lemma tells us that if we have a graph $G$ with large minimum degree we can find a large subset with the property that all vertices of $G$ have many neighbours inside this subset.
\begin{lem}\label{lem:subset}
Let $m' \ge m$. Any $m'$-vertex graph $G$ with minimum degree $\delta$ has an $m$-vertex subset $S$ with the property that every vertex of $G$ has more than $d\ge \frac m2$ neighbours in $S$, provided $mm'< r^{\floor{(r-1)d}}$ where $r:=\frac{m}{m'}/\frac{d}{\delta}\ge 1$.
\end{lem}
\begin{proof}
We note that we may assume that $d$ is an integer since otherwise we may reduce its value to its floor, as it is easy to see the main condition only relaxes by reducing $d$ and we replace the $2d \ge m$ assumption with $2d \ge m-1$. 

Let us choose an $m$-vertex subset uniformly at random. For convenience we will denote by $d'=\floor{\delta m/m'}\ge d+1$ (one should think of this as a lower bound on the expected number of neighbours, of any vertex of $G$, which got selected into our random subset), where the inequality follows since $d'-d=\floor{(r-1)d} \ge 1,$ as otherwise $\floor{(r-1)d}=0$ and our main assumption $mm'< r^{\floor{(r-1)d}}$ fails. 
Given a fixed vertex, the probability that at most $d$ among some fixed set of size $\delta$ of its neighbours got selected in the subset is at most 
\begin{align*}
\sum_{t=0}^{d} \binom{\delta}{t}\binom{m'-\delta}{m-t}/\binom{m'}{m}&\le
m\cdot \binom{\delta}{d}\binom{m'-\delta}{m-d}/\binom{m'}{m} \le m\cdot \binom{\delta}{d}/\binom{\delta}{d'}\cdot \binom{m'-\delta}{m-d}/\binom{m'-\delta}{m-d'}\\
&= m\cdot \prod_{i=1}^{d'-d}\frac{d+i}{m-d'+i} \cdot \prod_{i=1}^{d'-d}\frac{m'-\delta-(m-d)+i}{\delta-d'+i} \\
&\le m \cdot \left(\frac{dm'}{m(m'-\delta)}\right)^{d'-d} \cdot \left(\frac{m'-\delta}{\delta}\right)^{d'-d}\\
&= m \cdot \left(\frac{dm'}{\delta m}\right)^{d'-d}=m\cdot r^{-\floor{(r-1)d}}<1/m',
\end{align*}
where in the first inequality we used that $\binom{\delta}{t}\binom{m'-\delta}{m-t}$ is non-decreasing in $t$ for $t+1 \le d'$; in the second inequality we used that $\binom{m'}{m} =\sum_{t=0}^{\delta}\binom{\delta}{t}\binom{m'-\delta}{m-t} \ge \binom{\delta}{d'}\binom{m'-\delta}{m-d'}$; the third inequality follows since $\frac{d+i}{m-d'+i} \le \frac{d}{m-d'} \le \frac{dm'}{m(m'-\delta)}$ for any $i\ge 1$ since $m\le 2d+1 \le d+d'$
and $\frac{m'-\delta-(m-d)+i}{\delta-d'+i}\le \frac{m'-\delta-(m-d')}{\delta-\delta m/m'}\le \frac{m'-\delta-(m-\delta m/ m')}{\delta-\delta m/m'}=\frac{m'-\delta}{\delta}$
for any $1 \le i\le d'-d.$ This shows by an immediate union bound over all $m'$ vertices of $G$ that the desired subset exists.
\end{proof}

\section{A tiling lemma}

Given a graph $H$ we say that there is an \emph{$H$-tiling} of a graph $G$ if we can find vertex disjoint copies of $H$ in $G$ which cover all but at most $|G| \bmod{|H|}$ many vertices of $G$. When $|H| \mid |G|$ and we want to stress the point that an $H$-tiling covers all vertices of $G$ we will refer to it as a \emph{perfect $H$-tiling}. 

The goal of this section is to prove the following tiling lemma, which strengthens the classical Hajnal-Szemer\'edi Theorem \cite{H-S} when we know our graph has small independence number. 

\begin{lem}\label{lem:tiling}
Any graph $G$ with at least $4^{k+9}k^3$ vertices, with $\alpha(G) < k$ and $\delta(G) \ge \frac78 |G|$ admits a $K_k$-tiling.
\end{lem}

It is easy to see that in terms of the requirement on the size of $G$, one can not do better than exponential in $k$, by taking $G$ to be the random graph. 
It is easy to improve the $7/8$ factor in the minimum degree condition to any $1/2+\eps$ with $\eps>0$. This is best possible since one may take $G$ to be a vertex disjoint union of two cliques of almost the same size. We opt not to prove it in this optimal form for simplicity and since it is not necessary for our applications.    

A result along the same lines was obtained by Nenadov and Pehova \cite{N-P}, in fact they make a much weaker assumption on the independence number which results in the requirement for the ground graph to be significantly larger than the tiling graph (they make use of the Szemer\'edi regularity lemma), which is far from sufficient for us. Even if one only focuses on the absorption part of their argument, their approach requires too large a ground set. That said the overall structure of our argument is very similar to theirs, as well as to many other examples of the use of the absorption method. In fact, since we work with a rather strong assumption it makes for a great illustrative example for the method since most of the somewhat technical tools usually accompanying it, such as regularity lemma are not needed here.

Before turning to the proof we state, and prove for the sake of completeness, the following auxiliary proposition due to Montgomery \cite{richard-graph} stating the existence of sparse bipartite graphs, which have certain ``resilience'' properties with respect to matchings.

\begin{prop}\label{resilient-graph}
There exists a bipartite graph, with bipartition $(X \cup Y,Z)$ where $X$ and $Y$ are disjoint, $|X|=|Y|=2k$ and $|Z|=3k$ such that:
\begin{enumerate}[label=\alph*)]
    \item\label{itm:max-deg} Maximum degree is at most $40$.
    \item\label{itm:matching} Given any subset $X'\subset X$ with $|X'|=k$, there is a perfect matching between $X' \cup Y$ and $Z$.
\end{enumerate} 
\end{prop}

\begin{proof}

Fix some $Z_1 \subseteq Z$ of size $2k$ and place independently $20$ random perfect matchings between $Y$ and $Z_1$. Let the graph $Q_0$ be the union of these matchings. Given subsets $A\subset Y$ and $B\subset Z_1$, and a random matching $M$, the probability that \emph{all} vertices of $B$ get matched by $M$ to a vertex in $A$ is $\binom{|A|}{|B|}/\binom{2k}{|B|}$. Thus, for any $t\leq k/2$, the probability there is some set $B\subset Z_1$, of size $t$ with neighbourhood of size less than $2t$ in $Q_0$, is by a simple union bound over possible choices for $B$ and all possible choices for a set of size $2t$ containing its neighbourhood at most
\[
\binom{2k}{t}\binom{2k}{2t}\left(\binom{2t}{t}/\binom{2k}{t}\right)^{20}
\leq \left({2e^3}\left(\frac kt\right)^{3}\left(\frac{t}{k}\right)^{20}\right)^t
= \left(2e^3\left(\frac{t}{k}\right)^{17}\right)^t \le 2^{-11t}.
\]
Taking the sum over all $t \le k/2$ we see that the probability that such a $B$ of arbitrary size smaller than $k/2$ exists is at most $2^{-10}$. An identical calculation shows that the probability there is an $A\subset Y$ with $|A|\leq k/2$ and neighbourhood of size less than $2|A|$ is at most $2^{-10}$.
The probability that there are two subsets $A\subset Y$ and $B\subset Z_1$, with $|A|=|B|=\lceil k/2\rceil$ and no edges between $A$ and $B$ is at most
\[
\binom{2k}{\ceil{k/2}}^2\left(\binom{\floor{3k/2}}{\ceil{k/2}}/\binom{2k}{\ceil{k/2}}\right)^{20}\leq (4e)^{k}\left(\frac{3}{4}\right)^{10k}\leq 2^{-k/2}. 
\]
Since $2^{-1/2}+2^{-10}+2^{-10}<1$ there exists a graph $G$ satisfying all three of these properties. Duplicate each vertex in $Y$ to get $X$, and then duplicate $k$ of the vertices in $Z_1$ to get the set $Z_2$ and set $Z=Z_1 \cup Z_2$. We claim that this new bipartite graph with bipartition $(X \cup Y, Z)$ has the desired properties. For \ref{itm:max-deg}, since $Y \cup Z_1$ induces 20 matchings and we duplicated every vertex at most once the maximum degree is indeed at most 40.

Now, take any set $X'\subset X$ of size $k$. To show \ref{itm:matching} we will verify Hall's condition between $Z$ and $X' \cup Y$. For $A\subset Z$, let $A'$ be the larger of the sets $A\cap Z_1$ and $A\cap Z_2$, so that $|A|\geq|A'|\geq |A|/2$.

If $|A'|\leq k/2$. Then by the first property we ensured for $G$ we know $A'$ has at least $2|A'|\ge |A|$ neighbours inside $Y$, and hence so does $A$ inside $X' \cup Y$, as desired. Otherwise, $|A'|> k/2$ so by our second property of $G$ there can be at most $k/2$ vertices in $Y$ which are not adjacent to any vertex of $A'$. Since we obtained $X$ by duplicating vertices of $Y$ the same holds about non-neighbours of $A'$ inside $X$. In particular, all but at most $k$ vertices of $X' \cup Y$ are adjacent to some vertex in $A'$ so we are done unless $|A|>2k$.

By taking a subset of $A$ of size $2k$ we can see from the previous paragraph that the set $B$ of vertices in $X' \cup Y$, not adjacent to any vertex in $A$ has size at most $k$. By symmetry (note that $X'$ is simply a set of $k$ duplicates of vertices in $Y$, same as $Z_2$ is for $Z_1$) we can repeat the above argument with $B$ in place of $A$ to conclude there are at least $|B|$ vertices in $Z$ adjacent to some vertex in $B$. In particular, $|A| \le 3k-|B|,$ which is also the number of vertices adjacent to some vertex of $A$, as desired.
\end{proof}

We are now ready to prove \Cref{lem:tiling}.

\begin{proof}[ of \Cref{lem:tiling}]
The key part of the proof is establishing the existence of an absorbing subset of vertices of $G$. In particular, we want to find an $A\subseteq V(G)$, with the property that for any subset $R \subseteq V(G) \setminus A$ of size at most $4^k$ there is an $K_k$-tiling of $G[A\cup R].$

Before showing how to find such an $A$ let us demonstrate how to conclude the argument once we have it. Take a maximal vertex disjoint collection of $K_k$'s in $V(G)\setminus A$. Let $R$ be the set of remaining vertices in $V(G) \setminus A$, since it contains no copy of $K_k$ and no independent set of size $k$ we know $|R|\le r(K_k) \le 4^k$ where we used the classical upper bound on Ramsey numbers due to Erd\H{o}s and Szeker\'es. Now by the assumed absorbing property of $A$ we can $K_k$-tile $G[R \cup A]$ which together with our removed family of disjoint $K_k$'s gives us the desired $K_k$-tiling of $G$.

Now let us turn to finding our absorber $A$. We first show that for any set $S$ of $k$ vertices $\{s_1,\ldots,s_k\}$ one can find a ``local absorber'', which we will denote by $L_S$. The key property we require from $L_S$ is that we can perfectly $K_k$-tile both $L_S$ and $S \cup L_S$. We also want to be able to construct these absorbers to be disjoint for a collection of sets $S$, so we will show that we can find one avoiding any given set $F \supseteq S$ of at most $n/2$ vertices. Since $n/2 \ge 4^k \ge r(K_k)$ we can find a copy of $K_k$ among available vertices, we denote it as $K_S=\{w_1,\ldots,w_k\}$. Now for each $i$ we know $s_i$ and $w_i$ have at least $3/4 \cdot n$ common neighbours in $G$ so there are at least $n/4$ among the available vertices. Since $n/4-k^2 \ge 4^k \ge r(K_k)$ we can find a copy $C_i$ of $K_{k-1}$ inside their common neighbourhood in such a way that all $C_i$'s are disjoint and are disjoint from $K_S$. We now take $L_S$ to be the union of $K_S$ and $C_1,\ldots, C_k$ (see \Cref{fig:0.1} for an illustration). Note that if we wish to $K_k$-tile $L_S$ we know it consists of $k$ disjoint cliques $\{w_i\} \cup C_i$ of size $k$ (see \Cref{fig:0.3} for an illustration) and if we wish to tile $S \cup L_S$ it consists of $k+1$ disjoint cliques $K_S, \{s_i\} \cup C_i$ so this is indeed possible (see \Cref{fig:0.2} for an illustration).

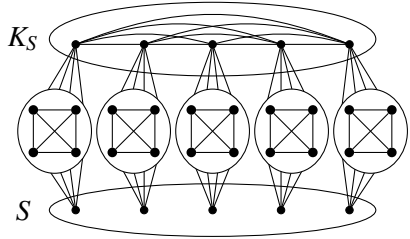
\begin{figure}[t]
\centering
\captionsetup{width=0.8\textwidth}
\begin{tikzpicture}[scale=0.7]
\defPt{0}{0}{x0}
\defPt{1.5}{0}{x1}
\defPt{3}{0}{x2}
\defPt{4.5}{0}{x3}
\defPt{6}{0}{x4}
\defPt{0.4}{1.65}{y0}
\defPt{1.7}{1.65}{y1}
\defPt{3}{1.65}{y2}
\defPt{4.3}{1.65}{y3}
\defPt{5.6}{1.65}{y4}
\defPt{0.4}{-1.5}{z0}
\defPt{1.7}{-1.5}{z1}
\defPt{3}{-1.5}{z2}
\defPt{4.3}{-1.5}{z3}
\defPt{5.6}{-1.5}{z4}

\fitellipsis{$(0,1.75)$}{$(6,1.75)$}{0.7};
\fitellipsis{$(0,-1.5)$}{$(6,-1.5)$}{0.5};

\foreach \i in {0,...,4}
{
\draw[] (y\i) -- ($(x\i)+(0.6,0)$);
\draw[] (y\i) -- ($(x\i)+(0.2,0)$);
\draw[] (y\i) -- ($(x\i)+(-0.2,0)$);
\draw[] (y\i) -- ($(x\i)+(-0.6,0)$);

\draw[] (z\i) -- ($(x\i)+(0.6,0)$);
\draw[] (z\i) -- ($(x\i)+(0.2,0)$);
\draw[] (z\i) -- ($(x\i)+(-0.2,0)$);
\draw[] (z\i) -- ($(x\i)+(-0.6,0)$);

\fitellipsis{$(x\i)-(0,0.7)$}{$(x\i)+(0,0.7)$}{0.7};
\pic[rotate=45, scale=0.4] at ($(x\i)$) {K4};
\draw[] (y\i) \smvx;
\draw[] (z\i) \smvx;
}


\node[] at ($(-0.6,-1.5)$) {$S$};
\node[] at ($(-0.6,1.75)$) {$K_S$};


\draw[] (y0) -- (y1) -- (y2) -- (y3) -- (y4);
\draw[] (y0) to[out=15,in=165] (y2) to[out=15,in=165] (y4);
\draw[] (y1) to[out=15,in=165] (y3);
\draw[] (y0) to[out=20,in=160] (y3);
\draw[] (y1) to[out=20,in=160] (y4);
\draw[] (y0) to[out=22,in=158] (y4);


\end{tikzpicture}
\caption{$L_S$: local absorber for $S$}
\label{fig:0.1}
\end{figure}

\begin{figure}[t]
\RawFloats
\begin{minipage}[t]{0.49\textwidth}
\centering
\captionsetup{width=\textwidth}
\begin{tikzpicture}[scale=0.7]
\defPt{0}{0}{x0}
\defPt{1.5}{0}{x1}
\defPt{3}{0}{x2}
\defPt{4.5}{0}{x3}
\defPt{6}{0}{x4}
\defPt{0.4}{1.65}{y0}
\defPt{1.7}{1.65}{y1}
\defPt{3}{1.65}{y2}
\defPt{4.3}{1.65}{y3}
\defPt{5.6}{1.65}{y4}
\defPt{0.4}{-1.5}{z0}
\defPt{1.7}{-1.5}{z1}
\defPt{3}{-1.5}{z2}
\defPt{4.3}{-1.5}{z3}
\defPt{5.6}{-1.5}{z4}

\fitellipsis{$(0,1.75)$}{$(6,1.75)$}{0.7};
\fitellipsis{$(0,-1.5)$}{$(6,-1.5)$}{0.5};

\foreach \i in {0,...,4}
{
\draw[] (y\i) -- ($(x\i)+(0.6,0)$);
\draw[] (y\i) -- ($(x\i)+(0.2,0)$);
\draw[] (y\i) -- ($(x\i)+(-0.2,0)$);
\draw[] (y\i) -- ($(x\i)+(-0.6,0)$);

\draw[color=red, line width=1.25pt] (z\i) -- ($(x\i)+(0.6,0)$);
\draw[color=red, line width=1.25pt] (z\i) -- ($(x\i)+(0.2,0)$);
\draw[color=red, line width=1.25pt] (z\i) -- ($(x\i)+(-0.2,0)$);
\draw[color=red, line width=1.25pt] (z\i) -- ($(x\i)+(-0.6,0)$);

\fitellipsis{$(x\i)-(0,0.7)$}{$(x\i)+(0,0.7)$}{0.7};
\pic[rotate=45, scale=0.4] at ($(x\i)$) {K4t};
\draw[] (y\i) \smvx;
\draw[] (z\i) \smvx;
}

\node[] at ($(-0.6,-1.5)$) {$S$};
\node[] at ($(-0.6,1.75)$) {$K_S$};

\draw[color=red, line width=1.25pt] (y0) -- (y1) -- (y2) -- (y3) -- (y4);
\draw[color=red, line width=1.25pt] (y0) to[out=15,in=165] (y2) to[out=15,in=165] (y4);
\draw[color=red, line width=1.25pt] (y1) to[out=15,in=165] (y3);
\draw[color=red, line width=1.25pt] (y0) to[out=20,in=160] (y3);
\draw[color=red, line width=1.25pt] (y1) to[out=20,in=160] (y4);
\draw[color=red, line width=1.25pt] (y0) to[out=22,in=158] (y4);

\foreach \i in {0,...,4}
{
\draw[] (y\i) \smvx;
}

\end{tikzpicture}
\caption{$K_5$-tiling of $L_S \cup S$}
\label{fig:0.2}
\end{minipage}\hfill
\begin{minipage}[t]{0.49\textwidth}
\centering
\captionsetup{width=\textwidth}
\begin{tikzpicture}[scale=0.7]
\defPt{0}{0}{x0}
\defPt{1.5}{0}{x1}
\defPt{3}{0}{x2}
\defPt{4.5}{0}{x3}
\defPt{6}{0}{x4}
\defPt{0.4}{1.65}{y0}
\defPt{1.7}{1.65}{y1}
\defPt{3}{1.65}{y2}
\defPt{4.3}{1.65}{y3}
\defPt{5.6}{1.65}{y4}
\defPt{0.4}{-1.5}{z0}
\defPt{1.7}{-1.5}{z1}
\defPt{3}{-1.5}{z2}
\defPt{4.3}{-1.5}{z3}
\defPt{5.6}{-1.5}{z4}

\fitellipsis{$(0,1.75)$}{$(6,1.75)$}{0.7};
\fitellipsis{$(0,-1.5)$}{$(6,-1.5)$}{0.5};

\foreach \i in {0,...,4}
{
\draw[color=red, line width=1.25pt] (y\i) -- ($(x\i)+(0.6,0)$);
\draw[color=red, line width=1.25pt] (y\i) -- ($(x\i)+(0.2,0)$);
\draw[color=red, line width=1.25pt] (y\i) -- ($(x\i)+(-0.2,0)$);
\draw[color=red, line width=1.25pt] (y\i) -- ($(x\i)+(-0.6,0)$);

\draw[] (z\i) -- ($(x\i)+(0.6,0)$);
\draw[] (z\i) -- ($(x\i)+(0.2,0)$);
\draw[] (z\i) -- ($(x\i)+(-0.2,0)$);
\draw[] (z\i) -- ($(x\i)+(-0.6,0)$);

\fitellipsis{$(x\i)-(0,0.7)$}{$(x\i)+(0,0.7)$}{0.7};
\pic[rotate=45, scale=0.4] at ($(x\i)$) {K4t};
\draw[] (y\i) \smvx;
\draw[] (z\i) \smvx;
}

\node[] at ($(-0.6,-1.5)$) {$S$};
\node[] at ($(-0.6,1.75)$) {$K_S$};

\draw[] (y0) -- (y1) -- (y2) -- (y3) -- (y4);
\draw[] (y0) to[out=15,in=165] (y2) to[out=15,in=165] (y4);
\draw[] (y1) to[out=15,in=165] (y3);
\draw[] (y0) to[out=20,in=160] (y3);
\draw[] (y1) to[out=20,in=160] (y4);
\draw[] (y0) to[out=22,in=158] (y4);

\end{tikzpicture}
\caption{$K_5$-tiling of only $L_S$}
\label{fig:0.3}
\end{minipage}
\end{figure}

We now proceed to specify our ``global'' absorber $A$. 
Let $\ell= \floor{\frac{n}{4^4k^2}}\ge 4^{k+5}k$. Let $X$ be a subset of $V(G)$ of size $2\ell$ with the property that any vertex of $G$ has at least $3\ell/2 =3/4 \cdot|X|$ neighbours in $X$. Existence of such an $X$ follows from \Cref{lem:subset} with $m'=n, m=2\ell, \delta \ge 7n/8$ and $d=3\ell/2$, so that $r \ge 7/6$ and its condition holds since $(7/6)^{\floor{\ell/4}} \ge 2^{n/(2\cdot 4^6k^2)} \ge n^2 \ge 2n\ell$.

We now take another, arbitrary, set of $2\ell$ vertices $Y$ and an arbitrary set $Z$ of size $3\ell(k-1)$ such that $X,Y,Z$ are disjoint. We further split $Z$ into $3\ell$ disjoint sets $Z_1,\ldots, Z_{3\ell}$ each of size $k-1$. We place each of $X,Y$ and $Z$ into our $A$. 

We now consider an auxiliary bipartite graph $D$, provided by \Cref{resilient-graph} with parts $X \cup Y$ and $\{Z_1,\ldots, Z_{3\ell}\}$, so with the property that every vertex has maximum degree $40$ but upon removing any $\ell$ vertices from $X$ the remainder has a perfect matching.

For every edge $e=(v,Z_i)$ of $D$ we place the local absorber $L_e:=L_{\{v\} \cup Z_i}$ avoiding any previously used vertices as well as $X,Y$ and $Z$ into $A$. We can always find it since there are at most $120\ell$ edges in $D$ and each $L_e$ we find has size $k^2$ so the total number of used vertices at any point is at most $4\ell+3\ell(k-1)+120\ell k^2 \le 127\ell k^2 \le n/2$. 

Finally, let us verify this $A$ has the desired property. Consider an $R \subseteq V(G) \setminus A$ of size at most $4^k$. Any $v \in R$, by our choice of $X$, has at least $3\ell/2 \ge 4^k$ neighbours in $X$, so we can find a copy of $K_k$ using $v$ and some $k-1$ vertices in $X$. In fact we can do so even if we disallow usage of up to $\ell$ vertices of $X$ since any vertex will still have $\ell/2\ge 4^k$ available neighbours left. So since $\ell \ge k|R|$ we can find a vertex disjoint collection of $K_k$'s which tiles $R$ and up to $\ell$ vertices of $X$. We then proceed to find disjoint copies of $K_k$ in the remainder of $X$ until we are left with at least $\ell$ and at most $\ell+k-1$ vertices, which we can since $\ell \ge 4^k$ so by Ramsey estimate we must still be able to find a new copy of $K_k$. At this point we discard the at most $k-1$ vertices to obtain a subset $X'$ of size exactly $\ell$. Since we discarded less than $|K_k|$ vertices, provided we can perfectly $K_k$-tile the rest of $A \cup R$ this will prove the result. Now, by the second property of our auxiliary graph $D$ there is a matching between $X' \cup Y$ and $\{Z_1,\ldots, Z_{3\ell}\}$. For any edge $e=(v,Z_i)$ of this matching we can $K_k$-tile $L_e \cup \{v\} \cup Z_i$ which all together gives us an $K_k$-tiling of $X'\cup Y \cup Z$. What is left to $K_k$-tile are $L_e$'s for any $e$ which was not a part of our matching which we can also do by the defining property of our local absorber. This completes the proof.
\end{proof}

\section{Ramsey number of copies}
In this section we will prove our main two results. We begin with the asymmetric and slightly simpler case, namely that of \Cref{thm:asym}.

\subsection{Asymmetric problem}\label{sec:asym}
We begin with the upper bound, which shall motivate the example for the lower bound as well.

\begin{thm}\label{thm:asym-ub}
Let $G$ and $H$ be graphs with $k=\max(|G|,|H|)$. Then, provided $n \ge 9^{2k+6}$, we have
		$$r(G,nH)\le n|H|+r(\mathcal{D}(G),H)-1.$$
\end{thm}
\begin{proof}
Let us consider a two coloured complete graph $K$ on $N=n|H|+r(\mathcal{D}(G),H)-1\ge n$ vertices. Let us assume towards a contradiction that $K$ does not contain a red $G$ or a blue $nH$.

Since there is no red copy of $G$ \Cref{lem:ramsey} with $d=9$ applied to the blue graph implies that there exists a subset of vertices $A$ with $m:=|A| \ge N/9^{k-1} \ge 9^{k+7}$ in which every vertex has blue degree at least $\frac{8}{9} m$. 

Now among the remaining vertices $V(K) \setminus A$ we take a maximal vertex disjoint collection of blue copies of $H$ and place their vertices in a set $B$. The set of remaining vertices, so $V(K) \setminus (A \cup B)=:C$, contains no red copy of $G$ nor a blue copy of $H$ so $|C|<r(G,H) \le 4^k$.

For any vertex $v$ of $C$ which has at least $4^k$ blue neighbours in $A$, we know that in its neighbourhood we can find a blue copy of $H$, so in particular we may replace one of the vertices in the copy by $v$ to obtain a blue $H$ using $v$ and some $k-1$ vertices of $A$. As long as we can find such a $v$ we take this blue copy of $H$ and add it to $B$. When we stop let $A_1$ be the set of the remaining vertices of $A$, so $|A_1|\ge |A|-|C|(k-1)$; let $B_1$ denote the new set $B$, which still has the property that we can perfectly tile its vertices with blue $H$; and let $C_1$ be the remaining vertices of $C$, which now each have less than $4^k$ blue neighbours in $A_1$.


In addition to not having any blue copies of $H$ we claim $C_1$ can not contain any member of $\mathcal{D}(G)$ in red. Indeed since $|A_1| \ge k4^k \ge (k-1)4^k + \alpha(G)$ if we could find a red member of $\mathcal{D}(G)$ in $C_1$ at most $(k-1)4^k$ vertices of $A_1$ send even a single blue edge to it. This implies we can find $\alpha(G)$ vertices in $A_1$ sending \emph{only} red edges to $C_1,$ which we can use to create a red copy of $G$. This implies $|C_1| \le r(\mathcal{D}(G),H)-1$. In particular, by our assumption on the size of $K$ we know $|A_1 \cup B_1| \ge n|H|$ so if we can find a blue $H$-tiling of $A_1 \cup B_1$ we obtain a contradiction. Since we already can perfectly tile $B_1$ it only remains to show that we can tile $A_1$ as well. We know that $|A_1|\ge |A|-4^k\cdot(k-1)\ge 9^{k+7}-4^k\cdot(k-1)\ge  4^{k+9}k^3$ and every vertex has minimum blue degree in $A_1$ at least $\frac{8}{9} \cdot m -4^k\cdot (k-1) \ge \frac78 \cdot |A_1|$ so \Cref{lem:tiling} applies and allows us to tile $A_1$ with blue copies of $H$, completing the proof.
\end{proof}

This establishes the upper bound in \Cref{thm:asym}. To see the lower bound, partition the vertex set of a complete graph on $n|H|+r(\mathcal{D}(G),H)-2$ vertices into sets $A$ of size $n|H|-1$ and $C$ of size $r(\mathcal{D}(G),H)-1$. We colour all edges inside $A$ blue, all edges between $A$ and $C$ red and we take a colouring of $C$ containing neither a red copy of any member of $\mathcal{D}(G)$ nor a blue copy of $H$. To see this colouring does not contain a red copy of $G$ note that the portion of $G$ inside $A$ must be an independent set of $G$, since $A$ spans an independent set in the red graph, so the remainder of $G$ contains a member of $\mathcal{D}(G)$ which would need to be embedded inside $C$. In terms of blue copies of $H$, since $H$ is connected and there are no blue edges between $A$ and $C$, each blue copy of $H$ is completely contained either in $A$ or in $C$. We know there are no copies in $C$ so the most we can find is $n-1$ in $A$ due to size considerations. This completes the proof of \Cref{thm:asym}.

Let us also briefly discuss the case of disconnected $H$. The answer becomes $n|H|+r(\mathcal{D}(G),\mathcal{C}(H))-1$, where $\mathcal{C}(H)$ is the family of connected components of $H$. The optimal colouring is similar as above noting that we may take $C$ to not contain any blue connected component of $H$ which allows one to finish the argument as above. For the upper bound the argument proceeds in the same way with an additional step once we reach the $A_1,B_1,C_1$ partition. As long as we can find a connected component of $H$ inside $C_1$ in blue we can find the rest of $H$ in $A_1$ and put the blue copy of $H$ into $B_1$. Once again since $|C_1|\le 4^k$ this can not happen many times and when it stops the remaining vertices of $C_1$ do not have any blue components of $H$ so there are at most $r(\mathcal{D}(G),\mathcal{C}(H))-1$ vertices left in $C_1$ and we can finish in the same way.

\subsection{Symmetric problem}\label{sec:sym}
In this subsection we will prove \Cref{thm:main} concerned with determining $r(nH)$.
As mentioned in the introduction the answer here is less explicit. Let us begin with an easy lower bound construction, due to Burr, Erd\H{o}s and Spencer \cite{bes} which while not being tight is going to be useful, as it allows us to estimate the size of the underlying graphs we are going to be working with. 

\begin{prop}\label{prop:lb}
Let $H$ be a graph with no isolated vertices then for any $n$
$$r(nH) \ge (2|H|-\alpha(H))n-1.$$
\end{prop}

\begin{proof}
Let $k=|H|$ and $\alpha=\alpha(H)$. Our goal is to give a $2$-colouring of a complete graph on $(2k-\alpha)n-2$ vertices without a monochromatic $nH$. Let us split its vertices into two sets $R$ and $B$ of sizes $(k-\alpha)n-1$ and $kn-1$ resp. We then colour all edges inside $R$ red, all edges inside $B$ blue and all edges in between red. Since $H$ has no isolated vertices any blue copy of $H$ must be completely inside $B$ which has size $n|H|-1$ so there is no blue $nH$. On the other hand every red copy of $H$ must use at least $k-\alpha$ vertices from $R$ as part of it we embed inside $B$ spans an independent set. Since $|R|=(k-\alpha)n-1$ there is no red $nH$, as desired.
\end{proof}

Since the part of the answer here that grows with $n$ is $(2|H|-\alpha(H))n$ just taking vertex disjoint copies of monochromatic $H$ as we did in the proof of \Cref{thm:asym-ub} does not lead us anywhere. Instead we will try to find a red copy of $H$ and a blue copy of $H$ which intersect in at least $\alpha(H)$ vertices. With this in mind an \emph{$H$-tie} is defined as a set of $2|H|-\alpha(H)$ vertices of a $2$-coloured complete graph which contains both a red copy of $H$ and a blue copy of $H$. This object plays a key role in understanding $r(nH)$ since as soon as we can find such an object if the rest of the graph contains a monochromatic $(n-1)H$ we can use the copy from our $H$-tie to extend it to a monochromatic $nH$. In particular, if $n$ is large enough to guarantee we can find an $H$-tie (provided there is no monochromatic $nH$) we know $r(nH) \le 2|H|-\alpha(H)+r((n-1)H)$. This in turn shows that $f(n)=r(nH)-(2|H|-\alpha(H))n$ becomes non-increasing when $n$ is large and is always at least $-1$ by \Cref{prop:lb}. Hence, $r(nH)=n(2|H|-\alpha(H))+c(H)$ when $n$ is large enough. It is not hard to show that one can find an $H$-tie even when $n$ is just exponential (provided there is no monochromatic $nH$), however determining when the long term behaviour begins, i.e.\ when $f(n)$ ultimately \emph{stops} decreasing and the final value $c(H)$ it settles in is a far more tricky proposition.

We will show that there exists a critical colouring, meaning that it avoids monochromatic $nH$ and has maximum number of vertices subject to this, with a lot structure. In particular, it will be very close to the colouring we used in the proof of the above lower bound, but with the additional ``exceptional'' small set. This will be the key step in the proof of \Cref{thm:main} since it reduces it to considering only these very special colourings. We note that we will assume $H$ has no isolated vertices for convenience and since one can easily deduce $r(nH)$ from $r(nH')$ where $H'$ is obtained from $H$ by removing all its isolated vertices.

\begin{thm}\label{thm:critical}
There exists $C>0$ such that the following holds for any $k$-vertex graph $H$ without isolated vertices. Provided $n \ge 2^{Ck}$, for some there exists a $2$-colouring of $K_{r(nH)-1}$, without a monochromatic $nH$ with the following structure: we can partition the vertex set into parts $R,B,E$ such that (see \Cref{fig:0.4} for illustration):
\begin{enumerate}
    \item All edges inside $R$ are red and all edges inside $B$ are blue.
    \item All edges between $R$ and $B$ are of the same colour, all edges between $R$ and $E$ are blue and all edges between $B$ and $E$ are red.
    \item\label{itm:3} There is no $H$-tie containing a vertex of $E$.
\end{enumerate}
\end{thm}

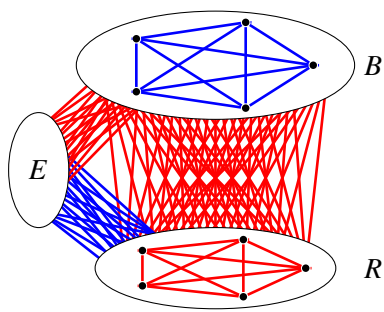
\begin{figure}[ht!]
\centering
\captionsetup{width=0.8\textwidth}
\begin{tikzpicture}[yscale=0.9,xscale=1, rotate=90]
\defPt{0}{0}{x0}
\defPt{3}{0}{x1}

\defPt{0}{-1.2}{y0}
\defPt{0}{-0.4}{y1}
\defPt{0}{0.4}{y2}
\defPt{0}{1.2}{y3}
\defPt{0}{-0.8}{y4}
\defPt{0}{-0}{y5}
\defPt{0}{0.8}{y6}

\defPt{3}{-1.5}{z0}
\defPt{3}{-0.5}{z1}
\defPt{3}{0.5}{z2}
\defPt{3}{1.5}{z3}
\defPt{3}{-1}{z4}
\defPt{3}{0}{z5}
\defPt{3}{1}{z6}

\defPt{0.85}{2.35}{a0}
\defPt{1.25}{2.35}{a1}
\defPt{1.65}{2.35}{a2}
\defPt{2.05}{2.35}{a3}

\defPt{-0.5}{0.8}{b0}
\defPt{-0.2}{0.8}{b1}
\defPt{0.2}{0.8}{b2}
\defPt{0.5}{0.8}{b3}

\defPt{2.5}{0.9}{c0}
\defPt{2.8}{0.9}{c1}
\defPt{3.2}{0.9}{c2}
\defPt{3.5}{0.9}{c3}

\foreach \i in {0,...,6}
    \foreach \j in {0,...,6}
        \draw[red, line width=1pt ] (y\i) -- (z\j);

\foreach \i in {0,...,3}
    \foreach \j in {0,...,3}
        \draw[blue, line width=1pt ] (a\i) -- (b\j);

\foreach \i in {0,...,3}
    \foreach \j in {0,...,3}
        \draw[red, line width=1pt ] (a\i) -- (c\j);

\fitellipsis{$(0.7,2.35)$}{$(2.2,2.35)$}{0.4};

\fitellipsis{$(x0)-(0,1.5)$}{$(x0)+(0,1.5)$}{0.6};
\pic[rotate=0, xscale=1.2, yscale=0.4,red] at ($(x0)$) {K5};

\fitellipsis{$(x1)-(0,1.75)$}{$(x1)+(0,1.75)$}{0.8};
\pic[rotate=0, xscale=1.3, yscale=0.6,blue] at ($(x1)$) {K5};

\node[] at ($(1.45,2.35)$) {$E$};

\node[] at ($(x0)+(0,-2.1)$) {$R$};
\node[] at ($(x1)+(0,-2.1)$) {$B$};

\end{tikzpicture}
\caption{The structure of a critical colouring. We do not know the colouring inside $E$ but it is strongly constrained by property \ref{itm:3}.}
\label{fig:0.4}
\end{figure}

Before turning to the proof let us note that \Cref{thm:critical} establishes the same structure of the colouring as introduced by Burr \cite{burr1}. That said, our proof does allow us to conclude some additional information on relative sizes of the sets involved, as will be discussed after the proof. 

Our strategy for proving \Cref{thm:critical} is going to be to start with an arbitrary critical colouring, gradually increase our understanding of its structure up to a point where we understand it well enough so that we can modify it into a colouring satisfying the properties in the statement but without creating a monochromatic $nH$. Let us now outline the steps of the proof, at each step we include a pointer to the claim in the main proof where it is done formally, in full detail.

\begin{itemize}
    \item We will find two large disjoint sets $R$ and $B$ which have very high minimum degree in red and blue resp.\ and which can be tiled with red $(k-\alpha,k)$-joins with red part in $R$ and blue in $B.$ See \Cref{claim:prop}.
    \item Outside of $R \cup B$ we set aside a maximal collection of $H$-ties and call its vertex set $T$ and denote the set of remaining vertices by $E$.
    \item We next show that $|E|$ must be small or we can find a monochromatic $nH$ (with the key observation being that we can tile the remainder of the graph with $H$-ties). See \Cref{claim:E}.
    \item We now move any $H$-tie in $R \cup B \cup E$ which contains a vertex of $E$ to $T$. With the key observation being that since $E$ is small this will not impact too badly the properties of $R$ and $B$ that we ensured.
    \item By moving an additional small set of $H$-ties to $T$ we can ensure all the edges between $R$ and $E$ are blue and all the edges between $B$ and $E$ are red. See \Cref{clm:4}. To achieve this we are using \Cref{claim:tie-finding} as a key tool which roughly speaking says that any vertex which does not send mostly blue edges to $R$ and red to $B$ can be made part of an $H$-tie and hence moved to $T$.
    \item Finally, we show that there are only few vertices in $R$ which have high blue degree towards $B,$ as otherwise we would be able to find a blue $nH.$ See \Cref{clm:5}. We then move all vertices with high blue degree to $T$ using the same $H$-tie moving strategy as before. 
\end{itemize}
At this point we are finally ready to introduce the new colouring satisfying the desired properties (see \Cref{clm:6}) and showing that if one can find a monochromatic $nH$ in the new colouring the same holds for the original (which is a contradiction).

\begin{proof}[ of \Cref{thm:critical}] 
Let $k=|H|$ and $\alpha=\alpha(H)$. Let us consider an arbitrary $2$-colouring of the complete graph $K$ on $N=r(nH)-1$ vertices not containing a monochromatic copy of $nH$. \Cref{prop:lb} implies that 
$N \ge n(2k-\alpha)-2\ge nk+n(k-\alpha)-2\ge nk+n-2$, where we used that $H$ is non-empty (since it has no isolated vertices) which implies $\alpha<k$. In particular, this implies that if we remove a maximal vertex disjoint collection of red $H$ the remainder has size at least $n$ and has no red copy of $H$. Remove some $n/2$ of the remaining vertices and apply \Cref{lem:ramsey} with $d=2^7$ to the blue graph to find a subset $B_0$ of size $m \ge \frac{n}{2^{7k}}$ which has minimum blue degree at least $(1-2^{-7})m$. Similarly taking a maximal collection of blue $H$ the remainder has at least $n$ vertices, so excluding the at most $n/2$ vertices already in $B_0$ we are left with at least $n/2$ of them which do not contain a blue $H$. So again \Cref{lem:ramsey} tells us we can find a subset $R_0$, disjoint from $B_0$, of size $m' \ge \frac{n}{2^{7k}}$ with minimum red degree at least $(1-2^{-7})m'$.

For certain technical reasons it is going to be helpful for us to ensure that both red and blue sets are of the same size. Let us w.l.o.g.\ assume that $m \le m'$, set $B_1=B_0$ and let $R_1$ be the subset of $R_0$ of size $m$ inducing a red subgraph of minimum degree at least $(1-2^{-6})m$. Such an $R_1$ exists by \Cref{lem:subset} applied with $d=(1-2^{-6})m$ and $\delta \ge (1-2^{-7})m'$ so that $r \ge 127/126$ and the requirement of the lemma $mm'< (127/126)^{\floor{m/128}}$ is clearly satisfied since $n/2 \ge m'\ge m \ge n/2^{7k}$. In order to preserve symmetry we will also relax our slightly stronger condition on minimum blue degree in $B_1$ to say that every vertex in $B_1$ has blue degree at least $(1-2^{-6})m$.

Now since $R_1$ still contains no blue $H$ and $B_1$ no red $H$, as well as $|B_1|,|R_1|=m \ge 2^{5k}$ we can use \Cref{lem:join} to find a $(k,k)$-join with red part in $R_1$ and blue part in $B_1$. We then take it out and repeat as long as we can. When we stop we have a collection of $(k,k)$-joins which cover the same number of vertices, namely at least $m-2^{5k}$ in each of $R_1$ and $B_1$. Now keep only the joins of the more common colour, say w.l.o.g. red. The number of vertices they cover in each of $R_1$ and $B_1$ is hence at least $m/2-2^{5k-1}$. Let $B_2$ be the set of vertices of these joins inside $B_1$ so $|B_2| \ge m/2-2^{5k-1}$. Finally, we will obtain $R_2$ by taking the union of $k-\alpha$ vertices taken uniformly at random from each join, from its red part inside $R_1$. 

In the following claim we collect all the properties of $R_2$ and $B_2$ that we will require, as well as explain more formally how $R_2$ is chosen.
\begin{claim}\label{claim:prop}
There exist disjoint sets of vertices $R_2$ and $B_2$ satisfying:
\begin{enumerate}[label=\alph*)]
    \item\label{prop:1} $|R_2|=(k-\alpha)m_0$ and $|B_2|=km_0$ for some $m_0 \ge n/2^{O(k)}.$
    \item\label{prop:2} $R_2$ contains no blue copy of $H$ and has minimum red degree at least $\frac{15}{16}|R_2|$. 
    \item\label{prop:3} $B_2$ contains no red copy of $H$ and has minimum blue degree at least $\frac{15}{16}|B_2|$
    \item\label{prop:4} We can tile $R_2 \cup B_2$ with red $(k-\alpha,k)$-joins each with red and blue parts inside $R_2$ and $B_2$ resp.
\end{enumerate}
\end{claim}

\begin{proof_claim}
We construct $B_2$ as above and will show that if $R_2$ is selected randomly as described above it satisfies the required red minimum degree condition of \ref{prop:2} with positive probability. 

Most of the properties are immediate from construction. For example \ref{prop:1} follows since we constructed $B_2$ as a disjoint union of $m_0$ sets of size $k$, where $m_0$ is the number of joins we had; $R_2$ was obtained as union of $m_0$ sets of size $k-\alpha$ and we ensured $km_0=|B_2| \ge m/2-2^{5k-1}$ so $m_0 \ge m/(2k)-2^{5k-1}/k \ge n/2^{O(k)}$. Part \ref{prop:4} follows since removing any $\alpha$ vertices from the red part of a red $(k,k)$-join creates a red $(k-\alpha,k)$-join. First parts of \ref{prop:2} and \ref{prop:3} follow since $R_2 \subseteq R_0$ and $B_2 \subseteq B_0$ and we constructed these sets to be blue and red $H$-free resp. For the second part of \ref{prop:3} note that since we removed exactly $m-|B_2|$ vertices from $B_1$ the minimum blue degree in $B_2$ is at least $$(1-2^{-6})m-(m-|B_2|)=|B_2|-m/2^6\ge \frac{31}{32}|B_2|-2^{5k-6} \ge \frac{15}{16}|B_2|$$
where we used $m \le 2|B_2|+2^{5k}$ and \ref{prop:1}. 
Finally, we need to verify the second part of \ref{prop:2}. Here, before we made our random choices we had the same property, namely the set $R'$ consisting of unions of the red parts of our $(k,k)$-joins had size $m_0$ and minimum red degree $\frac{31}{32}|R'|-2^{5k-6}\ge \frac{15.25}{16}|R'|$. Now our goal is to show that after making our random choices the probability  that for a fixed vertex we have chosen less than $\frac{15}{16}|R_2|$ of its neighbours is less than $\frac{1}{|R'|}=\frac{1}{km_0}$ since then a union bound ensures there exists an $R_2$ in which this happens for no vertex of $R'$. Fix a vertex $v$ in $R'$ and let $X_1,\ldots, X_{m_0}$ be the random variables counting the number of its red neighbours in each of the $m_0$ joins. Observe that $X_i$ are independent since our choices in different joins were independent, that $0\le X_i \le k-\alpha$ and that $\sum_{i=1}^{m_0} \E X_i \ge \frac{15.25}{16}|R'|\cdot \frac{k-\alpha}{k}=\frac{15.25}{16}(k-\alpha)m_0,$ 
since any vertex individually has probability exactly $\frac{k-\alpha}{k}$ of being selected. By Azuma--Hoeffding inequality \cite[Theorem 2]{azuma} we now have 
$$\P\left(\sum_{i=1}^{m_0} X_i \le \frac{15}{16}(k-\alpha)m_0\right)\le 2e^{-\frac{2\cdot 2^{-12}(k-\alpha)^2m_0^2}{m_0(k-\alpha)^2}}=2e^{-m_0/2^{11}}<\frac{1}{km_0},$$
as desired.
\end{proof_claim}

Armed with these $R_2$ and $B_2$ let $T$ be a maximal collection of vertex disjoint $H$-ties that we can find outside of $R_2 \cup B_2$ and denote the set of remaining vertices outside of $R_2 \cup B_2$ as $E_2$. Our next few steps will gradually ``clean-up'' the picture. Namely, we will identify a small (of size at most $2^{O(k)}$) problematic subset then repeatedly find disjoint $H$-ties within $R_2 \cup B_2 \cup E_2$ which together cover all the problematic vertices. Then we move these $H$-ties to $T$. In order to help keep track of the sizes, we note that at any point we will remove at most $2^{O(k)}$ vertices from $R_2$ or $B_2$ so the sets of vertices remaining there are still quite large since $|R_2|,|B_2|\ge m_0 \ge n/2^{O(k)}$ and in particular $n/2^{O(k)}-2^{O(k)}\ge 2^{O(k)}.$ Furthermore, we will also know that the minimum degree remains big, e.g. in the remainder of $R_2$ in red $\frac{15}{16}|R_2|-2^{O(k)}\ge \frac{7}{8}|R_2|.$ We will usually specify the exponent we need in the errors to make it easier to track where it comes from.

The following claim will be behind the fact that we will only ever need to remove at most $2^{O(k)}$ $H$-ties.
\begin{claim}\label{claim:E}
$|E_2|\le 2^{7k}$.
\end{claim}
\begin{proof_claim}
Towards a contradiction let us assume otherwise, so that $|E_2| > 2^{7k}$. Note first that if we can cover at least $\frac{k}{2k-\alpha}|E_2|$ vertices of $E_2$ with copies of $H$ of some fixed colour then we can find $n$ copies of $H$ in this colour. This follows since we can always cover a proportion of $\frac{k}{2k-\alpha}$ of both $R_2 \cup B_2$ and $T$ with $H$'s of either colour, since by \Cref{claim:prop} \ref{prop:4} we can cover $R_2 \cup B_2$ with $H$-ties and in any $H$-tie we can find a copy of $H$ of either colour which has $k$ vertices while the whole tie has $2k-\alpha$ vertices. This means that in the whole of $K$ we can find at least  $\frac{N}{2k-\alpha}\ge n-\frac{2}{2k-\alpha}>n-1$ disjoint copies of $H$ in our colour, where we used \Cref{prop:lb} which tells us $N \ge n(2k-\alpha)-2$, as claimed. This allows us to find a subset $R'$ of $E_2$ of size at least $\frac{k-\alpha}{2k-\alpha}|E_2|-k\ge \frac{|E_2|}{2k}-k \ge 2\cdot 2^{5k}$ which contains no blue copy of $H$. Similarly there is a subset $B'$ of $E_2$ of size at least $2 \cdot 2^{5k}$ which contains no red copy of $H$. By splitting vertices between $R'$ and $B'$ if necessary we can make $R'$ and $B'$ disjoint by losing at most another factor of $1/2$ in their size. This allows us to use \Cref{lem:join} to find a $(k,k)$-join which in particular contains an $H$-tie, this is a contradiction to how we constructed $E_2$.
\end{proof_claim}

Now for as long as we can find an $H$-tie in $R_2 \cup B_2 \cup E_2$ which uses at least $1$ vertex from $E_2$ we take it out and add it to $T$. We denote by $R_3,B_3$ and $E_3$ the sets of remaining vertices in $R_2,B_2$ and $E_2$ resp. Since $|E_2| \le 2^{7k}$ we know that $|R_3| \ge |R_2|-2k2^{7k}, |B_3| \ge |B_2|-2k2^{7k}$ and $|E_3|\le |E_2| \le 2^{7k}$. 

The following claim will allow us to impose quite a bit of structure on the colouring between $E_3$ and $R_3 \cup B_3$. 

\begin{claim}\label{claim:tie-finding}
Given a vertex $v$ which has $2^{5k}$ neighbours of colour $c_1$ in $R_3$ and $2^{5k}$ neighbours of colour $c_2$ in $B_3$, unless $c_1$ is blue and $c_2$ is red there exists an $H$-tie using $v$ and some of these neighbours.
\end{claim}
\begin{proof_claim}
Let $R'$ be the set of $2^{5k}$ neighbours of $v$ in colour $c_1$ in $R_3$ and $B'$ the set of its neighbours in colour $c_2$ in $B_3$. By \Cref{lem:join} we can find a $(k,k)$-join with red part in $R'$ and blue part in $B'$. W.l.o.g. we may assume the join is red. Now if $c_2$ is blue we can take $v$, $k-1$ vertices from $B'$ and $k-\alpha$ vertices from $R'$ to get our $H$-tie, since it still contains a blue $K_k$ and since $\alpha\le k-1$ we can still find a red $(k-\alpha,\alpha)$ join, which contains a red $H$. In the remaining case both $c_1$ and $c_2$ are red. Here we can replace a vertex from the red part with $v$ and take $k-\alpha$ vertices from $B'$ and they still make a red $(k,k-\alpha)$-join which makes an $H$-tie.
\end{proof_claim}

The next claim gets us a step closer to our desired structure.

\begin{claim}\label{clm:4}
There are sets $R_4 \subseteq R_3$ and $B_4 \subseteq B_3$ such that $|R_4|\ge |R_2|-2^{14k}$ and $|B_4| \ge |B_2|-2^{14k}$; there are no red edges between $E_3$ and $R_4$ and no blue edges between $E_3$ and $B_4$ and $(R_3 \setminus R_4) \cup (B_3 \setminus B_4)$ can be covered by disjoint $H$-ties.
\end{claim}
\begin{proof_claim}
The previous claim immediately implies that no vertex $v$ of $E_3$ can have red degree larger than $2^{5k}$ in $R_3$ or blue degree larger than $2^{5k}$ towards $B_3$, since otherwise taking whichever colour of its neighbours is most common in the other part we find an $H$-tie using $v$, all of which we previously removed. Therefore, there are at most $2^{5k}|E_3|\le 2^{12k}$ vertices in $R_3$ which send a blue edge towards $E_3$ and at most this many vertices in $B_3$ which send a red edge towards $E_3$. 

Let $v$ be one of these, at most $2^{12k}$ vertices, in $B_3$. We know its blue degree in $B_3$ is at least $\frac{15}{16}|B_2|-2k2^{7k}\ge 2k2^{12k}+2^{5k}$, where we used \Cref{claim:prop} \ref{prop:3}, the fact that $B_3$ is obtained from $B_2$ by removing up to $2k2^{7k}$ vertices and the last inequality follows from \Cref{claim:prop} \ref{prop:1}. Since $|R_3| \ge |R_2|-2k2^{7k} \ge 2k2^{12k}+ 2 \cdot 2^{5k}$ taking the majority colour neighbours of $v$ in $R_3$ and its blue neighbours in $B_3$ using \Cref{claim:tie-finding} we can find an $H$-tie inside $R_3 \cup B_3$ containing $v$. We move this tie to $T$ and repeat within the remaining sets until there are no more red edges between the remainder of $B_3$ and $E_3$. We only need to remove at most $2^{12k}$ vertices so at most this many ties get removed, since we had an excess of $2k2^{12k}$ in our degrees before the process, at any point during the process we have at least $2^{5k}$ left, as needed by \Cref{claim:tie-finding}. We can similarly remove any of the at most $2^{12k}$ vertices in $R_3$ which send a red edge towards $E_3$. We now take $R_4$ and $B_4$ to be the remaining sets of vertices inside $R_3$ and $B_3$ resp.

Since we removed at most $2^{12k}$ ties in both of our steps above and each has size at most $2k$ we have $|R_4| \ge |R_3|-2k\cdot 2\cdot 2^{12k}\ge |R_2|-2k \cdot 2^{7k}-4k 2^{12k}\ge |R_2|-2^{14k}$
 and similarly $|B_4|\ge |B_2|-2^{14k}$.
\end{proof_claim}

The claim provides us with new sets $R_4 \subseteq R_3$ and $B_4 \subseteq B_3$ and while our set $E_3$ did not change, for consistency let us set $E_4:=E_3$. Let $S$ denote the set of vertices in $R_4$ which have at least $k2^{15k}$ blue neighbours in $B_4$. We first show that if $S$ is big then we can find a blue $nH$. 

\begin{claim}\label{clm:5}
$|S| \le 2^{15k}$. 
\end{claim}
\begin{proof_claim}
Towards a contradiction assume $|S| > 2^{15k}$. We will show that we can find at least $2^{15k}$ vertex disjoint blue copies of $H$ using exactly one vertex from $S$ and having remaining vertices in $B_4$. We will find these copies one by one, at any point we have used up at most $(k-1)2^{15k}$ vertices from $B_4$. So, any vertex $v\in S$ still has at least $2^{15k}$ available blue neighbours inside $B_4$ and since $B_4$ has no red copy of $H$ so by Ramsey's theorem there is a blue copy of $H$. Replacing some vertex in this copy with $v$ we find our new copy of $H$ and can continue until we find our collection of $2^{15k}$ blue copies of $H$. The remainder of $B_4$ has minimum blue degree at least $\frac{15}{16}|B_2|-2^{14k}-(k-1)2^{15k} \ge \frac78 |B_2|$ so by \Cref{lem:tiling} we can tile the remainder with blue copies of $H$. In particular, taking into account that the tiling might not be perfect, we have covered at least $|B_4|-(k-1)+2^{15k}\ge |B_4|+2^{14k}+2^{7k}\ge |B_2|+|E_4|$ vertices of $R_4 \cup B_4 \cup E_4$ with blue copies of $H$, where in the second inequality we used \Cref{clm:4,claim:E}. Noting that $|B_2|=km_0$ makes a proportion of at least $\frac{k}{2k-\alpha}$ out of $R_4 \cup B_4 \subseteq R_2 \cup B_2$ we know that $|B_2|+|E_4|$ makes at least $\frac{k}{2k-\alpha}$ proportion of $R_4 \cup B_4 \cup E_4$. Since, as usual we can cover the same proportion of $T$ with blue copies of $H$, this means that as before we get at least $\frac{N}{2k-\alpha}>n-1$ vertex disjoint blue copies of $H$, which is a contradiction.
\end{proof_claim}

Now in a similar manner as before, since we know each member of $S$ has at least $2k\cdot2^{15k}+2^{5k}$ red neighbours in $R_4$ we can move at most $2^{15k}$ vertex disjoint $H$-ties to $T$ which together contained all vertices of $S$. Let $R_5$ and $B_5$ be the sets of remaining vertices of $R_4$ and $B_4$ resp. and again we set $E_5:=E_4$. We have $|R_5|\ge |R_2|-2^{14k}-2k\cdot2^{15k} \ge |R_2|-2^{17k}$ and $|B_5|\ge |B_2|-2^{17k}$ and the new bit of structure we obtained is that every vertex of $R_5$ has at most $2^{16k}$ blue neighbours in $B_5$ so that the vast majority of edges between $R_5$ and $B_5$ is red. See \Cref{fig:0.5} for illustration.

\begin{figure}[ht!]
\centering
\captionsetup{width=0.8\textwidth}
\begin{tikzpicture}[scale=0.5,rotate=90]

\defPt{0}{0}{x0}
\defPt{3}{0}{x1}

\defPt{0}{-1.2}{y0}
\defPt{0}{-0.4}{y1}
\defPt{0}{0.4}{y2}
\defPt{0}{1.2}{y3}
\defPt{0}{-0.8}{y4}
\defPt{0}{-0}{y5}
\defPt{0}{0.8}{y6}

\defPt{3}{-1.5}{z0}
\defPt{3}{-0.5}{z1}
\defPt{3}{0.5}{z2}
\defPt{3}{1.5}{z3}
\defPt{3}{-1}{z4}
\defPt{3}{0}{z5}
\defPt{3}{1}{z6}

\defPt{0.85}{2.35}{a0}
\defPt{1.25}{2.35}{a1}
\defPt{1.65}{2.35}{a2}
\defPt{2.05}{2.35}{a3}

\defPt{-0.5}{0.8}{b0}
\defPt{-0.2}{0.8}{b1}
\defPt{0.2}{0.8}{b2}
\defPt{0.5}{0.8}{b3}

\defPt{2.5}{0.9}{c0}
\defPt{2.8}{0.9}{c1}
\defPt{3.2}{0.9}{c2}
\defPt{3.5}{0.9}{c3}

\defPt{0.6}{-15}{g0}
\defPt{0.6}{-9}{g1}
\defPt{0.6}{-5}{g2}

\foreach \i in {0,...,6}
    \foreach \j in {0,...,6}
        \draw[red, dotted, line width=1pt ] (y\i) -- (z\j);

\foreach \i in {0,...,3}
    \foreach \j in {0,...,3}
        \draw[blue, line width=1pt ] (a\i) -- (b\j);

\foreach \i in {0,...,3}
    \foreach \j in {0,...,3}
        \draw[red, line width=1pt ] (a\i) -- (c\j);

\fitellipsis{$(0.7,2.35)$}{$(2.2,2.35)$}{0.4};

\fitellipsis{$(x0)-(0,1.5)$}{$(x0)+(0,1.5)$}{0.6};
\pic[rotate=0, xscale=0.6, yscale=0.25,red,dashed] at ($(x0)$) {K5};

\fitellipsis{$(x1)-(0,1.75)$}{$(x1)+(0,1.75)$}{0.8};
\pic[rotate=0, xscale=0.65, yscale=0.35,blue,dashed] at ($(x1)$) {K5};

\node[] at ($(1.45,2.35)$) {\footnotesize $E_5$};

\node[] at ($(x0)+(-1.5,0)$) {$R_5$};
\node[] at ($(x0)+(-1.5,-3)$) {$R$};
\node[] at ($(x1)+(1.5,0)$) {$B_5$};
\node[] at ($(x1)+(1.5,-3)$) {$B$};
\node[] at ($(g0)+(1,-2)$) {$T$};

\node[] at (2.3,-12) {$\cdots$};
\node[] at (0.5,-12) {$\cdots$};
\foreach \i in {0,...,2}
\pic[scale=0.5,rotate=90] at ($(g\i)$) {K55};

\fitellipsisnfill{$(g0)-(-1,2.8)$}{$(g2)+(1,2)$}{2.8};
\fitellipsisbfill{$(g0)-(-1.8,2)$}{$(g2)+(2.5,7.8)$}{1.1};
\fitellipsisrfill{$(g0)-(0,2)$}{$(g2)+(-0.7,7.8)$}{1.1};

\end{tikzpicture}
\caption{Structure of our original colouring. Tiny exceptional set $E_5=E$ is joined completely in red/blue to $B_5/R_5$. $R_5/B_5$ are almost completely red/blue. Edges in between them are mostly red. $T$ is a union of $H$-ties. $R=R_5$ union $k-\alpha$ vertices from each tie in $T$. $B=B_5$ union $k$ vertices from each tie in $T$. }
\label{fig:0.5}
\end{figure}
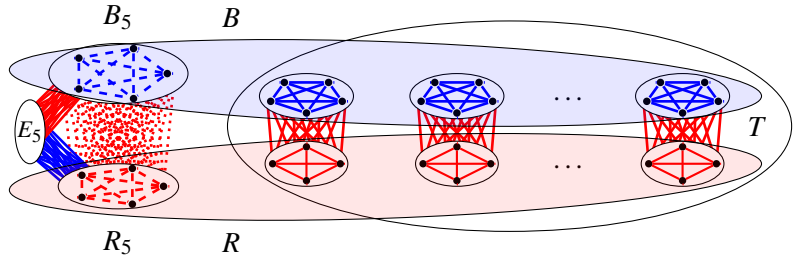

We now obtained enough structure to implement our strategy. Let us define a new colouring $c'$ of our complete graph. We will introduce a partition of the vertex set $R,B$ and $E$, where $E=E_5$ is the exceptional set we obtained above and we take $R$ to be the union of $R_5$ and $k-\alpha$ vertices from each $H$-tie in $T$ while we take $B$ to be a union of $B_5$ and the remaining $k$ vertices from each $H$-tie in $T$. In $c'$ we keep the same colouring of $E=E_5$ as in our original colouring; we colour $R$ completely red; $B$ completely blue and all edges in-between $R$ and $B$ red; in-between $E$ and $R$ blue and in-between $E$ and $B$ red. This means that the colouring has the desired form and it remains to check that there is no $H$-tie using any vertices of $E$ in $c'$ and that if it contains a monochromatic copy of $nH$ then so did our original colouring. To avoid confusion we note that $R,B$ and $E$ make a partition defining our new colouring while $R_5,B_5,E_5$ and $T$ make a partition with a lot of structure in our old colouring.

\begin{claim}\label{clm:6}
There is no $H$-tie in $c'$ which uses at least one vertex of $E$.
\end{claim}
\begin{proof_claim}
Suppose towards a contradiction that there exists  an $H$-tie in $c'$ using some vertices of $E.$ Since it consists of at most $2k$ vertices and all vertices of $R$ are the same and all vertices of blue are the same our $H$-tie needs to be a subset of $E$ and a red $(2k,2k)$-join with red part in $R$ and blue part in $B$. So provided we can find in our original colouring a red $(2k,2k)$-join with red part in $R_5$ and blue part in $B_5$ this would imply that we could find an $H$-tie using the same subset of vertices from $E$ in our original colouring as well (since we do know the colour of edges between $E$ and both $R_5$ and $B_5$), this is a contradiction. To see such a join exists note that we can find a red $K_{2k}$ in $R_5$ since it is blue $H$-free (by \Cref{claim:prop} since $R_5 \subseteq R_2$) and has size at least $2^{3k}$.
Since there are at most $2k\cdot 2^{16k}$ blue neighbours of some vertex in this clique the remaining $|B_5|-2k\cdot 2^{16k}\ge 2^{3k}$ vertices of $B_5$ are joined completely in red to our red $K_{2k}$. Since they are red $H$-free (by \Cref{claim:prop} since $B_5 \subseteq B_2$) we can find a blue $K_{2k}$ which completes our desired join.
\end{proof_claim}

Let us now turn to showing there is no monochromatic $nH$ in $c'$. Let us assume, towards a contradiction, that there is a monochromatic $nH$ in $c'$. Observe that among the $n$ copies of $H$ making this $nH$ there are at most $|E|\le 2^{7k}$ (by \Cref{claim:E} and $E \subseteq E_2$) copies which intersect $E$. We claim that we can find the same number of copies of $H$ in our original colouring, which in addition use the same number of vertices in both $R_5$ and $B_5$ as they do in $R$ and $B$ in the new colouring. This will be immediate from the following claim.

\begin{claim}
We can find $2^{7k}$ disjoint red $(k,k)$-joins with red part in $R_5$ and blue part inside $B_5$.
\end{claim}

\begin{proof_claim}
We find our joins one by one, at any point we have used up at most $k2^{7k}$ vertices so among the remaining $|R_5|-k2^{7k} \ge 2^{2k}$ vertices, since $R_5$ is blue $H$-free we can find a red $K_k$. The $k$ vertices of this clique send at most $k2^{16k}$ blue edges towards $B_5$, so they have at least $|B_5|-k2^{16k}-k2^{7k}\ge 2^{2k}$ red common neighbours among still available vertices of $B_5$, among which we can find a blue $K_k$, once again since $B_5$ is red $H$-free.
\end{proof_claim}

Now assign for each of our $c'$-monochromatic copies of $H$ which intersect $E$ one of the $(k,k)$-joins we found above in our original colouring. Looking at any copy of $H$ note that on the union of $H \cap E$ and vertices of its assigned join both our original and new colouring are exactly the same, so we can find the same (in terms of sizes in each part) copy of $H$ in our old colouring, which we refer to as the corresponding copy. 

Let us now remove all copies of $H$ intersecting $E$ in our new colouring and all our corresponding copies of $H$ in our old colouring. In the old colouring let $R_6$ and $B_6$ be the sets of remaining vertices from $R_5$ and $B_5$ and let $t$ be the number of $H$-ties making $T$ at this point. Note that $|B_6| \ge |B_5|-2k\cdot2^{7k}\ge |B_2|-2^{17k}-2k\cdot2^{7k} \ge |B_2|-2^{18k}$ and similarly $|R_6| \ge |R_2|-2^{18k}$. Now, all the remaining copies of $H$ from our monochromatic $nH$ in $c'$ belong to $R \cup B$ and in particular taking into account the number of vertices used in the removed copies they have only $r=|R_6|+t(k-\alpha)$ vertices of $R$ and $b=|B_6|+tk$ vertices of $B$ at their disposal. We know that on these vertices we can find at most $\floor{\frac{r}{k-\alpha}}$ red copies of $H$ since any red copy of $H$ in $R \cup B$ must use at least $k-\alpha$ vertices of $R$ since its part in $B$ is an independent set. Since $H$ has no isolated vertices any blue copy of $H$ in $R \cup B$ must be entirely inside $B$ so we can find at most $\floor{\frac{b}{k}}$ disjoint blue copies of $H$.

If we can show that we can find this many red and blue disjoint copies of $H$ in $R_6 \cup B_6 \cup T$ in our original colouring we complete the proof since this would imply there was a monochromatic copy of $nH$ in our original colouring as well. First observe that we can find exactly $t$ copies in either colour in $T$ so the task reduces to finding $\floor{\frac{|R_6|}{k-\alpha}}$ disjoint red copies and $\floor{\frac{|B_6|}{k}}$ disjoint blue copies of $H$ in $R_6 \cup B_6$. Note that $B_6$ induces a blue graph with minimum degree at least $\frac78 |B_6|$, this is since we obtained it from $B_2$ by deleting (several times) at most an exponential number of vertices, minimum degree in $B_2$ was $\frac{15}{16}|B_2|$ and $|B_2|/16 \ge n/2^{O(k)}$ (using \Cref{claim:prop}) is bigger than any exponential number of deleted vertices. \Cref{lem:tiling} tells us we can tile it with blue $H$ which gives precisely the desired number of copies. In the red case similarly $R_6$ still induces a red graph with minimum degree at least $\frac78|R_6|$ and we can tile it with red copies of $K_{k-\alpha}$ and our remaining task is to show we can find disjoint sets of $\alpha$ red common neighbours in $B_6$ for each of our red $K_{k-\alpha}$. This is easily achieved since there are at most $(k-\alpha)k2^{15k}\le2^{16k}$ vertices in $B_6$ sending a blue edge towards a fixed set of vertices of our red $K_{k-\alpha}$ in $R_6$ and at any point we use up at most $$\frac{|R_6|}{k-\alpha}\cdot \alpha\le m_0 \alpha \le m_0k-2^{18k}-2^{16k}-\alpha \le |B_6|-2^{16k}-\alpha,$$ vertices of $B_6$, where in the first inequality we used that $R_6\subseteq R_2$ and $|R_2|=(k-\alpha)m_0$, in the second inequality we used $m_0(k-\alpha) \ge m_0$ which is bigger than any exponential function in $k$ and in the final inequality we used that $|B_6| \ge |B_2|-2^{18k}=m_0k-2^{18k}.$ So we always have new $\alpha$ vertices of $B_6$ which are joined completely in red to our current $K_{k-\alpha}$. So, indeed we can find our collection of $\floor{\frac{|R_6|}{k-\alpha}}$ red $H$'s, completing the proof. 
\end{proof}

Observe that our proof gives a lot of additional information on the relations between sizes of $R,B$ and $E$. For example, we know that 
\begin{equation}\label{eq}
|R|,|B| \ge m_0-2^{O(k)} \ge n/2^{O(k)}-2^{O(k)} \ge k2^{7k} \ge k|E|,
\end{equation} since they contain $R_6$ and $B_6$ resp.\ which we obtain by removing up to $2^{O(k)}$ vertices from $R_2,B_2$ both of which have size at least $m_0$ and using \Cref{claim:prop,claim:E}. We also know\footnote{Under the (w.l.o.g.) assumption the edges in between $R$ and $B$ are red.} that $|B| \ge \frac{k}{2k-\alpha}N-2^{O(k)}\ge kn-2^{O(k)}\ge(k-1)n \ge \alpha n$, since $B$ upon removal of $2^{O(k)}$ vertices covers a proportion of $\frac{k}{2k-\alpha}$ of the graph and using \Cref{prop:lb} to lower bound $N$. It is also not hard to recover these bounds directly from the properties we state. For example, instead of using \Cref{claim:E} we can directly obtain an even stronger bound on $|E|$. Observe that $E$ may not contain any monochromatic member of $\mathcal{D}'(H)$, defined as the collection of graphs we can obtain\footnote{Recall here that $\mathcal{D}(H)$ stands for the family of subgraphs of $H$ obtained by removing \emph{maximal} independent sets so in particular $\mathcal{D}'(H)\subseteq \mathcal{D}(H)$.} by removing an independent set of $H$ of maximum possible size: $\alpha(H)$. To see this note that if $E$ contained a red $H'\in \mathcal{D}'(H)$ so with $k-\alpha$ vertices then taking this $H'$ and $k$ vertices from $B$ would create an $H$-tie using vertices of $E$. Similarly, no such blue $H'$ can exist in $E$. This implies $|E|\le r(\mathcal{D}'(H))\le r(H) \le 4^k.$ 

So if we treat $H$ to be of constant size it is fairly easy to check all possible colourings of sets $E$ consisting of up to $4^k$ vertices and to determine the largest number of monochromatic copies of $H$ we can find, by running an exhaustive search. While technically this strategy allows one to determine $r(nH)$ whenever $n$ is large enough so that \Cref{thm:critical} applies, it might seem that in practice it is not of much use. Be that as it may it turns out that in many interesting cases deducing $r(nH)$ from \Cref{thm:critical} is actually fairly straightforward. For example, this is the case when $H$ is a clique. In addition we have the following estimates which often turn out to be tight, either in the sense that upper and lower bounds match, or that the argument behind it is easily strengthened to show either the upper or the lower bound is tight. In order to state it we will need to define two more families related to $\mathcal{D}(H)$. The first one $\mathcal{D}_c(H)$ is the family of subgraphs of $H$ obtained by removing a maximal independent set of $H$ and taking a connected component of the remainder. Similarly, $\mathcal{D}'_c(H)$ is the family of subgraphs of $H$ obtained by removing a \emph{maximum} independent set of $H$ and taking a connected component of the remainder. For example, if $H=C_6$ we have $\mathcal{D}(H)=\{2K_2,3K_1\}, \mathcal{D}'(H)=\{3K_1\}, \mathcal{D}_c(H)=\{K_2,K_1\}$ and $\mathcal{D}_c'(H)=\{K_1\}.$

\begin{thm}\label{thm:estimate}
Let $H$ be a connected $k$-vertex graph and $n\ge 2^{O(k)}$ then
    $$r(\mathcal{D}_c(H),\mathcal{D}(H))-2 \le r(nH)-(2|H|-\alpha(H))n\le r(\mathcal{D}_c'(H),\mathcal{D}(H))-2.$$
\end{thm}
\begin{proof}
Let us begin with the lower bound. We will take a $2$-colouring of a complete graph on $(2k-\alpha)n+r(\mathcal{D}_c(H),\mathcal{D}(H))-3$ vertices whose vertex set we partition into sets $R,B$ and $E$ of size $(k-\alpha)n-1$, $kn-1$ and $r(\mathcal{D}_c(H),\mathcal{D}(H))-1$ resp. We colour all edges inside $R$ red and inside $B$ blue, all edges between $R$ and $B$ red, $E$ and $B$ red and $E$ and $R$ blue, while we colour $E$ in such a way to avoid any red member of $\mathcal{D}_c(H)$ and any blue member of $\mathcal{D}(H)$.
Observe first that there are no blue copies of $H$ which use any vertices from $E \cup R$ since $H$ is assumed to be connected and $E$ contains no member of $\mathcal{D}(H)$. Since there are $n|H|-1$ vertices in $B$ we can not fit $n$ disjoint blue copies there. In terms of red copies, the portion of any red copy of $H$ inside $E$ may be obtained by removing an independent set of $H$ (its part inside $B$) and is disconnected from its part in $R$. This means that its part inside $E$, provided it is non-empty, contains a member of $\mathcal{D}_c(H)$. So any red copy of $H$ must be completely contained in $R \cup B$ and hence needs to use at least $k-\alpha$ vertices of $R$. Since $|R|<n(k-\alpha)$ this means there can be no red copy of $nH$. 

Turning to the upper bound, let us apply \Cref{thm:critical} to obtain a partition of the vertex set of the complete graph on $r(nH)-1$ vertices into sets $R,B$ and $E$ and a colouring of its edges satisfying the properties stated in the theorem. First, as long as we can find a blue copy of $H$ which uses any vertices from $E$ we take it out, when we stop, the remaining set $E'$ contains no blue member of $\mathcal{D}(H)$ as it together with some vertices of $R$ would give another blue copy of $H$ and we always have enough vertices of $R$ remaining since by \eqref{eq} $|R|\ge k|E|$. The removed vertices together with all vertices from $B$ can span at most $n|H|-1$ vertices or they would contain a blue $nH$. Observe also that $|R|\le (k-\alpha)n-1$ or we can find a red copy of $nH.$ Note also that there can be at most $kn-1$ vertices in $B \cup (E \setminus E')$ since it consists of a collection of disjoint blue copies of $H$ which contain $E \setminus E'$ so if we had more than $nk$ vertices in this set we could use the leftover part of $B$ to find a blue $nH$. Since we are decomposing the complete graph on $r(nH)-1$ vertices into $R, B \cup (E \setminus E')$ and $E'$ this implies that unless the desired upper bound on $r(nH)$ hold $E'$ must have size at least $r(\mathcal{D}(H),\mathcal{D}_c'(H)).$  Since we argued above it does not contain a blue member of $\mathcal{D}(H)$ so it must contain a red $H'\in \mathcal{D}_c'(H)$. But then taking $k-\alpha-|H'|$ vertices of $R$ and $k$ of $B$ we find an $H$-tie intersecting $E$, a contradiction. 
\end{proof}
The bounds in the above result may seem fairly similar and it is indeed not at all trivial to find an example for which the bounds do not match. For example, the bounds match for any graph $H$ in which every maximal independent set is of maximum size. As a concrete example when $H=K_k$ we have $\mathcal{D}(H)=\mathcal{D}_c(H)=\mathcal{D}'_c(H)=\{K_{k-1}\}$ which means the above theorem gives \Cref{thm:clique} as a corollary.

Finally, let us formally conclude \Cref{thm:main} from \Cref{thm:critical}.

\begin{proof}[ of \Cref{thm:main}]
Let $n_0$ be the smallest value of $n$ for which \Cref{thm:critical} applies. Our goal is to show that for any $n>n_0$ we have $r(nH)=r((n-1)H)+2k-\alpha,$ where as usual $k=|H|$ and $\alpha=\alpha(H)$.

Let $c_n$ be the $nH$ critical colouring guaranteed by \Cref{thm:critical} with defining sets $R,B$ and $E$. We may assume w.l.o.g.\ that edges between $R$ and $B$ are red. As discussed previously (see \eqref{eq}) we may also assume $|R|,|B| \ge k$. We now claim that if we remove $k-\alpha$ vertices from $R$ and $k$ vertices from $B$ we obtain a colouring without a monochromatic $(n-1)H$. This is indeed immediate since if we could find a monochromatic $(n-1)H$ then we can extend it into a monochromatic copy of $nH$ using the removed vertices, since they spanned a red $(k-\alpha,k)$-join which contains both a red and a blue copy of $H$. This implies $r((n-1)H)\ge r(nH)-(2k-\alpha)$.

Let now $c_{n-1}$ be the $(n-1)H$ critical colouring guaranteed by \Cref{thm:critical} with defining sets $R',B'$ and $E'$. We may assume w.l.o.g.\ that edges between $R'$ and $B'$ are red. As discussed previously (see \eqref{eq}) we may also assume $|E'|\le 2^{O(k)},$ as well as $|R'|,|B'| \ge k |E'|$ and $|B'| \ge \alpha n+2^{O(k)}$. We now claim that if we add $k-\alpha$ vertices to $R'$ and $k$ vertices to $B'$ we obtain a colouring without a monochromatic $nH$. To see this suppose otherwise and assume we can find a monochromatic $nH$. Assume first the colour is blue. Note that at most $|E'|\le 2^{O(k)}<n$ of the copies can use a vertex of $E'$ and any other copy must be completely contained within the blue part. Removing $k$ vertices making such a copy of $H$ implies we had a blue $(n-1)H$ in $c_{n-1}$, which is a contradiction (we are using the fact all the vertices inside the blue part are joined to all other vertices in the same way). If the colour was red, similarly we must be able to find a copy disjoint from $E'$, which hence must use at least $k-\alpha$ vertices from the red part. Again removing the $k$ vertices making this copy of $H$ gives a red copy of $(n-1)H$ in $c_{n-1}$ (we are using here the fact this $(n-1)H$ can use at most $2^{O(k)}+(n-1)\alpha$ vertices from the blue part). This shows $r(nH)\ge r((n-1)H)+2k-\alpha$.
\end{proof}

\section{Concluding remarks}
Answering the question of Burr, Erd\H{o}s and Spencer \cite{bes} from 1975, in this paper we showed that the Ramsey number of the copy graph $nH$ settles in its long term behaviour, already when $n$ is exponential in $|H|$. As we already discussed in the introduction, this is essentially best possible in general. That said a conditional improvement may still be possible, namely it could be interesting to show that the long term behaviour occurs already when $n \ge O(r(H))$, since for graphs $H$ with smaller $r(H)$ this would be a significant improvement over our result. Our arguments definitely seem to offer a lot of potential here. Perhaps a good starting point would be to show something of the sort is true in the somewhat easier asymmetric case of $r(G,nH)$ where there is only one place in our proof of \Cref{thm:asym} which really requires an exponential number of vertices.

While we decided to focus on the original, classical instance of the problem raised already by Burr, Erd\H{o}s and Spencer in order to present our new ideas more clearly our methods are fairly robust. For example, a natural generalisation is to consider $r(G,H)$ where $G$ and $H$ are taken to be disjoint unions of $n$ and $m$ graphs taken from some finite families $\mathcal{G}$ and $\mathcal{H}$, respectively. This was considered by Burr in \cite{burr2}, where he established a result along the lines of our \Cref{thm:critical}, but with a much stronger requirement on how large $\min(n,m)$ should be. Our methods should allow one to obtain similar, essentially optimal improvement in this setting as well. This together with only assuming a lower bound on $\max(n,m)$ has the makings of a nice student project. 

A final further direction considers an analogous question in other settings. A particularly nice setting that comes to mind is that of tournaments. Here, a nice starting point may be the natural analogue of \Cref{thm:clique} which asks for how big a tournament you need to take in order to find $n$ vertex disjoint copies of a transitive tournament on $k$ vertices, denoted $\text{Tr}_k$. Here the answer is that the vertex set of any sufficiently large tournament may be decomposed into vertex disjoint $\text{Tr}_k$. This result is attributed to Erd\H{o}s (see \cite{reid}) and independently to Lonc and Truszcy\'nski \cite{Lonc}. The next question is how big is ``sufficiently large''? Erd\H{o}s' argument needs $4^k$ vertices while a more general approach of Lonc and Truszcy\'nski requires a double exponential number. Erd\H{o}s' argument was matched by Caro \cite{caro} who strengthened the general approach of \cite{Lonc}. On the other hand, the best known result, due to Erd\H{o}s and Moser from 1964 \cite{erdos-moser}, allowing one to find even a single copy of $\text{Tr}_k$ requires at least about $2^{k}$ vertices. After the first version of this paper appeared on Arxiv Zach Hunter observed that one can easily modify \Cref{lem:tiling} by replacing cliques with transitive tournaments in order to improve the bounds of Erd\H{o}s and Caro and show that $k^{O(1)}2^{k}$ vertices already suffice to tile with $\text{Tr}_k$. This might be viewed as a single colour analogue of our Ramsey problem which would be trivial in the undirected case, so perhaps a more fitting analogue would be to ask what happens if the tournament is $2$-coloured and we consider the so called oriented Ramsey numbers (see \cite{oriented-ramsey} for more details and definitions). Another interesting direction is the case of induced Ramsey numbers (see \cite{axen} for more details and definitions).


\section*{Acknowledgments} 
We would like to thank Shoham Letzter, Zach Hunter and Domagoj Brada\v{c}, as well as the anonymous referees, for a careful reading of this paper and a number of useful comments and suggestions.

\bibliographystyle{amsplain}


\begin{aicauthors}
\begin{authorinfo}[mbuc]
  Matija Buci\'c\\
  Institute for Advanced Study, School of Mathematics and\\ Princeton University, Department of Mathematics\\
  Princeton, USA\\ 
  matija.bucic\imageat{}ias\imagedot{}edu \\
  \url{https://www.math.ias.edu/~matija.bucic/}
\end{authorinfo}
\begin{authorinfo}[johan]
  Benny Sudakov\\
  ETH Z\"urich, Department of Mathematics\\
  Z\"urich, Switzerland\\
  benjamin.sudakov\imageat{}math\imagedot{}ethz\imagedot{}ch \\
  \url{https://people.math.ethz.ch/~sudakovb/}
\end{authorinfo}
\end{aicauthors}

\end{document}